\documentclass[a4paper,11pt]{article}
\usepackage{amsmath}
\usepackage[latin1]{inputenc}
\usepackage{graphics}
\usepackage{verbatim}
\usepackage{listings}
\usepackage{dsfont}
\usepackage{amsmath}
\usepackage{amsfonts}
\usepackage{amssymb}
\usepackage{graphicx}
\usepackage{geometry}
\usepackage{graphics}
\usepackage{url}
\usepackage{color}
\usepackage{textcomp}
\usepackage[normalem]{ulem}

\newtheorem{set2}{Satz}[section]
\newtheorem{theorem}[set2]{Theorem}

\newtheorem{corollary}[set2]{Corollary}

\newtheorem{definition}[set2]{Definition}

\newtheorem{lemma}[set2]{Lemma}
\newtheorem{notation[set2]}{Notation}

\newtheorem{remark}[set2]{Remark}

\newenvironment{hypotheses}[1][Assumptions]{\textbf{#1}}{}
\newcommand{\ep}{\hfill{$\square$}}
\newenvironment{proof}[1][Proof]{\textbf{#1.} }{\
\\}

\def\XXint#1#2#3{{\setbox0=\hbox{$#1{#2#3}{\int}$}
\vcenter{\hbox{$#2#3$}}\kern-.5\wd0}}

\newcommand{\1}{\mathbf{1}}
\newcommand{\bl}{\left(}
\newcommand{\br}{\right)}
\newcommand{\e}{\varepsilon}

\newcommand{\di}{\,\mathrm{div}}
\newcommand{\R}{\mathbb{R}}
\newcommand{\N}{\mathbb{N}}
\newcommand{\C}{\mathcal}
\newcommand{\OL}{\overline}
\newcommand{\UL}{\underline}
\newcommand{\dx}{\,\mathrm dx}
\newcommand{\dy}{\,\mathrm dy}
\newcommand{\dt}{\,\mathrm dt}
\newcommand{\ds}{\,\mathrm ds}

\newcommand{\diota}{\,\mathrm d\iota}
\newcommand{\dxt}{\,\mathrm dx\,\mathrm dt}

\newcommand{\dxs}{\,\mathrm dx\,\mathrm ds}

\newcommand{\esssup}{\mathop\mathrm{ess\,sup}}

\newcommand{\CC}{\textbf{C}}
\newcommand{\DD}{\textbf{D}}
\definecolor{purple}{rgb}{0.6,0,1}

\allowdisplaybreaks

\date{\today}

\title{Damage processes in thermoviscoelastic materials with damage-dependent thermal expansion coefficients}
\author{Christian Heinemann\footnote{Weierstrass Institute for Applied
Analysis and Stochastics, Mohrenstr.~39, D-10117 Berlin,
Germany. E-mail: \textit{christian.heinemann@wias-berlin.de}\newline
Most of the work has been done during a visit of C.H. to the Dipartimento di Matematica of the Universit\`a deli Studi di Pavia in January and February 2014.
}, Elisabetta Rocca\footnote{Weierstrass Institute for Applied
Analysis and Stochastics, Mohrenstr.~39, D-10117 Berlin,
Germany. E-mail: \textit{elisabetta.rocca@wias-berlin.de} and Dipartimento di Matematica ``F. Enriques'',
Universit\`{a} degli Studi di Milano, Milano I-20133, Italy.
E-mail: \textit{elisabetta.rocca@unimi.it}.  The author
is supported by the FP7-IDEAS-ERC-StG
Grant \#256872 (EntroPhase) { and by GNAMPA (Gruppo Nazionale per l'Analisi Matematica, la Probabilit\`a e le loro Applicazioni) of INdAM (Istituto Nazionale di Alta Matematica).}}}

\begin{document}
%\selectlanguage{english}
%\noindent
\maketitle

\begin{abstract}
In this paper we prove existence of global in time weak solutions for a highly nonlinear 
PDE system arising in the context of damage phenomena in thermoviscoelastic materials.
The main novelty of the present contribution  with respect to the ones already present in the literature 
consists in the possibility of taking into account a damage-dependent thermal expansion coefficient. 
This term implies the presence of nonlinear coupling terms in the PDE system, which make the analysis more challenging.  
\end{abstract}

\vspace{4mm}

\noindent {\bf AMS (MOS) subject clas\-si\-fi\-ca\-tion:}
35D30,  34B14, 74A45.

\section{Introduction}
  
	We consider the PDE system, in $\Omega\times (0,T)$, where $\Omega\subseteq \R^d$ (with $d\in\{1,2,3\}$) is a bounded and sufficiently regular domain and $T$ denotes a final time,
	\begin{subequations}
	\label{eqn:pde}
	\begin{align}
%		&\textit{Heat equation:}\notag\\
		\label{eqn:pde1}
		&\qquad\textsf{c}(\theta)\theta_t-\di(\textsf{K}(\theta)\nabla\theta)+\rho(\chi)\theta\di(u_t)+\theta\chi_t+\rho'(\chi)\theta\di(u)\chi_t=g,\\
%		&\textit{Balance of forces:}\notag\\
		\label{eqn:pde2}
		&\qquad u_{tt}-\di(b(\chi)\CC\e(u))-\di(a(\chi)\DD\e(u_t))+\di(\rho(\chi)\theta\1)=\ell,\\
%		&\textit{Evolution of the damage process:}\notag\\
		\label{eqn:pde3}
		&\qquad\chi_t+\xi+\varphi-\Delta_p\chi+\gamma(\chi)+\frac{b'(\chi)}{2}\CC\e(u):\e(u)-\theta-\rho'(\chi)\theta\di(u)=0
	\end{align}
	\end{subequations}
	with subgradients $\xi\in\partial I_{[0,\infty)}(\chi)$ and $\varphi\in\partial I_{(-\infty,0]}(\chi_t)$.
	The initial-boundary conditions are
	\begin{subequations}
	\begin{align}
		&\theta(0)=\theta^0,
		&&u(0)=u^0,
		&&u_t(0)=v^0,
		&&\chi(0)=\chi^0
		&&\text{in }\Omega,\\
		&\textsf{K}(\theta)\nabla\theta\cdot\nu=0,
		&&u=0,
		&&\nabla\chi\cdot\nu=0
		&&&&\text{on }\partial\Omega\times(0,T).
	\end{align}
	\end{subequations}
	
The state variables and unknowns of the problem are the absolute temperature $\theta$, whose evolution is ruled by the internal energy balance \eqref{eqn:pde1},
the vector of small displacements $u$, satisfying the momentum balance \eqref{eqn:pde2}, and the damage parameter $\chi$, representing the local proportion of
damage: $\chi=1$ means that the material is completely safe, while $\chi=0$ means it is completely damaged. Indeed, the two contraints $\chi\in [0,+\infty)$
and $\chi_t\leq 0$ together with the assumption  $\chi^0\in [0,1]$ imply that $\chi\in [0,1]$ during the evolution, as it results from its physical meaning.

It has to be mentioned that we use the \emph{small perturbation assumption} (cf. \cite{Germain73}) which neglects the quadratic contributions 
\begin{equation}
\label{diss}
a(\chi)\DD\e(u_t):\e(u_t)+|\chi_t|^2
\end{equation}
on the right-hand side
of \eqref{eqn:pde1} (see below). In consequence the system \eqref{eqn:pde1}-\eqref{eqn:pde3} is not thermodynamically consistent.
For an analysis of the full system with constant heat expansion coefficients we refer to \cite{RR14}.

The main novelty of this contribution consists in the possibility of taking into account the dependence of the thermal expansion coefficient $\rho\in C^1([0,1])$
on the damage variable $\chi$.
The progression of damage is accompanied with an increase and proliferation of micro-cavities and micro-cracks in the considered material
(as pointed out in engineering literature on damage; see, e.g., \cite[pp. 2-3]{LD05}).
This loss of structural integrity may also reduce the elastic response on temperature changes modeled by the heat expansion term in \eqref{eqn:pde2}.
But a dependence of $\rho$ on the damage also provokes the presence of two new nonlinear terms in \eqref{eqn:pde1} coupling them nonlinearly with both the momentum balance \eqref{eqn:pde2} and the damage evolution \eqref{eqn:pde3}.
Especially the coupling term $\rho'(\chi)\theta\di(u)\chi_t$ in \eqref{eqn:pde1} complicates the analysis and requires elaborate estimation techniques to gain the desired a priori estimates.
Moreover, a dependence on $\chi$ in the $u$-equation \eqref{eqn:pde2} appears explicitly as well as a further dependence on $u$ and $\theta$ in the $\chi$-equation \eqref{eqn:pde3}. 
The other two coeffiecients $\textsf{c}$ and $\textsf{K}$ appearing in equation \eqref{eqn:pde1} represent respectively the heat capacity and the heat conductivity of the system 
and will have to satisfy proper growth conditions (cf. Remark~\ref{rem:ck}), while the function $g$ denotes a given heat source. 

In the momentum balance \eqref{eqn:pde2} $\varepsilon(u):=(u_{i,j}+u_{j,i})/2$ denotes the linearized symmetric strain tensor, while the functions $b, \, a\in C^1([0,1])$ demarcate the damage dependence of the
elasticity and viscosity  modula, respectively. In the present contribution we will restrict to the case of incomplete damage, i.e. to the case where $a(x), b(x)\geq \eta>0$ (cf. \cite{RR12} for the
complete damage model in case $\rho=0$). The function $\ell$ on the right-hand side in \eqref{eqn:pde2} represents a given external force.
In order to avoid overburden the presentation, homogeneous Dirichlet data are assumed for $u$ (see Remark \ref{remark:homDirichlet} for further comment).

Finally, in the inclusion \eqref{eqn:pde3},  the selections $\xi$ and $\varphi$ of the two maximal monotone operators, acting on $\chi$ and $\chi_t$ respectively, are introduced in order to give the constraints on the 
damage parameter ($\chi\in [0,1]$ as soon as $\chi^0\in [0,1]$) and on the irreversibility of the damage process ($\chi_t\leq 0$). The $p$-Laplacian operator $\Delta_p\chi:=\mathrm{div}(|\nabla\chi|^{p-2}\nabla\chi)$ accounts for the nonlocal interactions between 
particles, but the restriction of the exponent $p>d$ is mainly due to analytical reasons. It is introduced, in particular, in order to obtain sufficient regularity on $\chi$ needed in \eqref{eqn:pde2} to obtain an 
enhanced estimate on $\epsilon(u)$, which appears at power $2$ in \eqref{eqn:pde3} and so it has to be estimated in a better space than $L^2(\Omega\times(0,T))$.
In addition to that, the enhanced regularity of $\chi$ enables the usage of approximation techniques in order to treat the doubly nonlinear inclusion \eqref{eqn:pde3} in a weak formulation.
Moreover, the function $\gamma$ is assumed to be smooth but possibly non monotone. 
We would like to emphasize that the weak formulation for the $\chi$-equation used in the sequel was introduced in \cite{WIAS1520}
(dealing with Cahn-Hilliard systems coupled with elasticity and damage processes; see also \cite{WIAS1722, WIAS1759,WIAS1569}).

In the remaining part of the Introduction we will briefly explain the derivation of \eqref{eqn:pde} referring to \cite{RR12} for more details.  
	
	The System \eqref{eqn:pde1}-\eqref{eqn:pde3} can be derived from fundamental balance laws in continuum mechanics
	supplemented with constitutive relations used to describe thermoviscoelastic solids.
	In this approach, we make use of the free energy $\C F$ given by \cite[Sec.\ 4.5, pp.\ 42-43]{Fr12}
	\begin{align}
	\label{eqn:freeEnergy}
		\C F(\theta,\e(u),\chi,\nabla\chi)
		={}&\int_\Omega\bl\frac 1p|\nabla\chi|^p+\widehat\gamma(\chi)+\frac{b(\chi)}{2}\CC\e(u):\e(u)\br\dx\\
		&+\int_\Omega\Big( f(\theta)-\theta\chi-\rho(\chi)\theta\di(u)+I_{[0,\infty)}(\chi)\Big)\dx
	\end{align}
	and the dissipation potential defined by
	\begin{align*}
		\C P_{\theta,\chi}(\nabla\theta,\chi_t,\e(u_t))=\int_\Omega\bigg(\frac{\textsf{K}(\theta)}{2}|\nabla\theta|^2
			+\frac 12|\chi_t|^2+\frac{a(\chi)}{2}\DD \e(u_t):\e(u_t)+I_{(-\infty,0]}(\chi_t)\bigg)\dx.
	\end{align*}
	For notational convenience, we write $\C P$ instead of $\C P_{\theta,\chi}$. Let us point out that the gradient of  $\chi$
accounts for the influence of damage at a material point,
undamaged in its neighborhood.
In this sense the term $\frac{1}{p}|\nabla\chi|^p$ models
nonlocality of the damage process and effects like 
possible hardening or softening (cf.\ also \cite{BMR09} for further comments on this topic).
Gradient regularizations of $p$-Laplacian type
are often adopted  in the mathematical papers on damage  (see for example \cite{bobo2,BS04,WIAS1520, Mielke06}),
and in the
 modeling literature as well
  (cf., e.g., \cite{FN96,Fre02}).
	
	Equation \eqref{eqn:pde1} is obtained from the internal energy balance which reads as
	$$
			e_t+\di q=g+\sigma:\e(u_t)+B\chi_t+H\cdot\nabla\chi_t,
	$$
	where $e$ denotes the internal energy, $q$ the heat flux,
	$g$ the heat source, $\sigma$ the stress tensor, $\e(u_t)$ the linearized strain rate tensor, $H$ and $B$ the so-called microscopic forces (cf. \cite{Fr12}).
	The quantities above are given by the following constitutive relations
	\begin{align*}
		&\sigma=\frac{\partial\C F}{\partial\e(u)}+\frac{\partial\C P}{\partial\e(u_t)},
		&&B\in\frac{\partial\C F}{\partial\chi}+\frac{\partial\C P}{\partial\chi_t},
		&&H=\frac{\partial \C F}{\partial\nabla\chi}+\frac{\partial\C P}{\partial\nabla\chi_t},\\
		&e=\C F-\theta\frac{\partial\C F}{\partial\theta},\qquad\qquad
		&&q=-\frac{\partial \C P}{\partial\nabla\theta}.
	\end{align*}
	Note that, for analytical reasons, 
	we have neglected the quadratic contributions\linebreak $a(\chi)\e(u_t):\CC\e(u_t)+|\chi_t|^2$ on the right-hand side of \eqref{eqn:pde1}, using the already mentioned
	small perturbation assumption.
In fact, to our knowledge only few results are available on
diffuse interface  models in thermoviscoelasticity (i.e.\ also accounting for the evolution of the
displacement variables, besides the temperature and the
order parameter): among others, we quote \cite{RR08, RR12} where the small perturbation assumption is adopted in case of constant $\rho$ and  \cite{Roubicek10} where
a PDE system coupling the momentum balance equation, the temperature equation (with quadratic nonlinearities)  and a \emph{rate-independent} flow rule for an internal dissipative variable $\chi$ (such as the damage parameter) has been
analyzed. Finally, a temperature-dependent, \emph{full} model
for (rate-dependent) damage has been addressed
in \cite{bobo2} as well, but only with local-in-time existence results.

 Moreover, we make use of the assumption
	\begin{align*}
	%\label{eqn:heatCapacity}
		\textsf{c}(\theta)= -\theta f''(\theta),
	\end{align*}
	where $f$ is a concave function.
	Eventually, the equation for the balance of forces \eqref{eqn:pde2} can be written as
	$$
		u_{tt}-\di\sigma=\ell
	$$
	with external volume forces $\ell$ and the evolution of the damage processes as described in equation \eqref{eqn:pde3} is derived from a
	balance equation of the microscopic forces, i.e.
	$$
		B-\di H=0.
	$$

	To handle non-constant heat capacities $\mathsf c$, we perform an enthalpy transformation of system \eqref{eqn:pde1}-\eqref{eqn:pde3}.
	To this end, we introduce the primitive $\widehat{\textsf{c}}$ of $\textsf{c}$ as
	\begin{align}
	\label{eqn:cHat}
		\widehat{\textsf{c}}(r):=\int_0^r \textsf{c}(\theta)\,\mathrm d\theta.
	\end{align}
	The enthalpy transformation of system \eqref{eqn:pde} yields
	\begin{subequations}
	\label{eqn:pdeTransf}
	\begin{align}
	\label{eqn:system1}
		&w_t-\di(K(w)\nabla w)+\Theta(w)\chi_t+\rho(\chi)\Theta(w)\di(u_t)+\rho'(\chi)\Theta(w)\di(u)\chi_t=g,\\
	\label{eqn:system2}
		&u_{tt}-\di(b(\chi)\CC\e(u))-\di(a(\chi)\DD\e(u_t))+\di(\rho(\chi)\Theta(w)\1)=\ell,\\
	\label{eqn:system3}
		&\chi_t+\xi+\varphi-\Delta_p\chi+\gamma(\chi)+\frac{b'(\chi)}{2}\CC\e(u):\e(u)-\Theta(w)-\rho'(\chi)\Theta(w)\di(u)=0.
	\end{align}
	\end{subequations}
	with $\xi\in\partial I_{[0,\infty)}(\chi)$ and $\varphi\in\partial I_{(-\infty,0]}(\chi_t)$ and the transformed quantities
	\begin{align}
	\label{eqn:wThetaKDef}
		&w:=\widehat{\textsf{c}}(\theta),
		&&\Theta(w):= \widehat{\textsf{c}}^{-1}(w),
		&&K(w):=\frac{\mathsf{K}(\Theta(w))}{\textsf{c}(\Theta(w))}.
	\end{align}
  
As already mentioned in the Introduction, the main difficulty here, with respect to the previous works in the literature, consists in the presence of the nonlinearities due to the fact that the temperature expansion term 
depends on $\chi$. Indeed, following \cite{Roubicek10, RR12}, here we will combine the conditions on $\mathsf{K}$ with conditions on
the heat capacity  coefficient $\mathsf{c}$ to handle the nonlinearities $\rho(\chi)\theta\di(u_t)$, $\theta\chi_t$, $\rho'(\chi)\theta\di(u)\chi_t$  in \eqref{eqn:pde1} by means of a so-called Boccardo-Gallou\"et type 
estimate on $\theta$. The reader
 may consult \cite{ZR66} for various examples in which a superquadratic growth in $\theta$ for the heat
conductivity $\mathsf{K}$ is imposed.

As for the triply nonlinear inclusion \eqref{eqn:pde3}, we will use a notion of solution derived in \cite{WIAS1520}.
The authors have devised a weak formulation consisting
of a \emph{one-sided} variational inequality (i.e.\ with test functions
having a fixed sign), and of an \emph{energy inequality}, see Definition~\ref{def:weakSolution} later.
Finally, let us notice that uniqueness of solutions remains an open problem even in the isothermal case. The main problem is, in general, the doubly nonlinear character of \eqref{eqn:pde3} (cf. also \cite{CV90} for examples 
of non-uniqueness in general doubly nonlinear equations).  

The paper is organized as follows.
In Section \ref{section:assumptions}, we list all assumptions which are used throughout this paper and introduce some notation.
Subsequently, a suitable notion of weak solutions for system \eqref{eqn:pde} as well as the main result, existence of weak solutions (see Theorem \ref{theorem:existence}),
are stated in Section \ref{section:notion}.
In the main part, the proof of the existence theorem is firstly performed for a truncated system in Section \ref{section:existence1} and finally for
the limit system in Section \ref{section:existence2}.
	
\section{Notation and assumptions}
\label{section:assumptions}
	Let $d\in\{1,2,3\}$ denote the space dimension.
	For the analysis of the transformed system \eqref{eqn:system1}-\eqref{eqn:system3}, the central hypotheses are stated below.\vspace*{0.5em}\\
	\begin{hypotheses}
	\begin{enumerate}
		\item[(A1)]
			$\Omega\subseteq\R^d$ is a bounded $C^2$-domain.
		\item[(A2)]
			The function $\Theta:\R\to\R$ is assumed to be Lipschitz continuous with $\Theta(w)\geq 0$ and $\Theta'(w)\geq 0$
			for a.e. $w\geq 0$ and should satisfy the growth condition
			\begin{align*}
%				c_1(w^{1/{\sigma_1}}-1)\leq
				\Theta(w)\leq c_0(w^{1/{\sigma}}+1)
			\end{align*}
			for all $w\geq 0$ and for constants $\sigma\geq 3$ and $c_0>0$.
			Moreover, we assume $\Theta(w)=0$ for all $w\leq 0$.
		\item[(A3)]
			The heat conductivity function $K:\R\to\R$ is assumed to be continuous and should satisfy the estimate
			\begin{align*}
				c_1(w^{2q}+1)\leq K(w)\leq c_2(w^{2q_0}+1)
			\end{align*}
			for all $w\geq 0$ and for constants $c_1,c_2,q,q_0>0$ satisfying
			$$
				1/\sigma\leq 2q-1,\qquad q\leq q_0<q+\frac 12.
			$$
		\item[(A4)]
			The damage-dependent potential function $\widehat \gamma$ is assumed to satisfy $\widehat \gamma\in C^1([0,1])$.
%		\item[(A5)]
%			The functions $f\in C^2(\R)$ (related to the heat capacity function $\textsl{c}$ by \eqref{eqn:heatCapacity}) is supposed to satisfy $f''(x)<0$ for all $x\in\R$.
		\item[(A5)]
			The coefficient functions $a\in C^1([0,1])$ and $b\in C^2([0,1])$ should satisfy the estimate $a(x),b(x)\geq\eta$ for all $x\in[0,1]$
			and a constant $\eta>0$.
		\item[(A6)]
			The 4th order stiffness tensor $\CC\in\C L(\R_\mathrm{sym}^{d\times d};\R_\mathrm{sym}^{d\times d})$
			is assumed to be
				symmetric and positive definite, i.e.
				\begin{align}
				\label{eqn:anisotropy}
					&\CC_{ijlk}=\CC_{jilk}=\CC_{lkij},&&e:\CC e\geq c_3|e|^2\text{ for all }e\in \R_\mathrm{sym}^{d\times d}
				\end{align}
				with constant $c_3>0$.
				The viscosity tensor is given by
				\begin{align}
				\label{eqn:viscousTensor}
					\DD=\mu\CC,
				\end{align}
				where $\mu>0$ is a constant.
%				of the form
%				\begin{align*}
%					&\CC\e=\DD\e=\lambda\tr(\e)\mathbf 1+2\mu\e,
%				\end{align*}
%				where $\mathbf 1$ denotes the identity tensor and $\lambda,\mu$ are the so-called Lam\'e parameters.
%			The 4th order stiffness tensors $\CC,\DD\in\C L(\R_\mathrm{sym}^{d\times d};\R_\mathrm{sym}^{d\times d})$
%			are assumed to be 
		\item[(A7)]
			The thermal expansion coefficient $\rho$ depending on $\chi$ is assumed to fulfill $\rho\in C^1([0,1])$.
		\item[(A8)]
			The constant $p$ (occurring in the $p$-Laplacian in \eqref{eqn:pde3} and in \eqref{eqn:system3}, respectively)
			should satisfy $p>d$.
	\end{enumerate}
	\end{hypotheses}
	\begin{remark}
			We would like to remark that condition \eqref{eqn:viscousTensor} is needed in order to perform the $W^{2,s}$-regularity argument
			in the discrete scheme (see Lemma \ref{lemma:discrSys}).
	\end{remark}
	\begin{remark}\label{rem:ck}
		The Assumptions (A2) and (A3) can also be formulated in terms of
		the original heat conductivity function $\mathsf{K}$ and the
		heat capacity function $\mathsf{c}$ as follows.
		\begin{enumerate}
			\item[(A2')]
				The function $\mathsf{c}$ should be continuous and should satisfy the estimate
				$$
					\widetilde c_0(\theta^{\sigma-1}+1)\leq\mathsf{c}(\theta)
				$$
				for all $\theta\geq 0$ and for constants $\sigma\geq 3$ and $\widetilde c_0>0$.
			\item[(A3')]
			The function $\mathsf{K}$ is assumed to be continuous and should satisfy the estimate
			\begin{align*}
				c_1(\widehat{\mathsf{c}}(\theta)^{2q}+1)\mathsf{c}(\theta)\leq \mathsf{K}(\theta)\leq c_2(\widehat{\mathsf{c}}(\theta)^{2q_0}+1)\mathsf{c}(\theta)
			\end{align*}
			for all $\theta\geq 0$ and for constants $c_1,c_2,q,q_0>0$ satisfying $1/\sigma\leq 2q-1$ and\linebreak $q\leq q_0<q+\frac 12$
			(see \eqref{eqn:cHat} for the definition of $\widehat{\mathsf{c}}$).
		\end{enumerate}
	\end{remark}
	\begin{remark}
		The growth assumptions (A2) and (A3) are necessary to obtain the a priori estimates (uniform with respect to the truncation parameter $M$)
		of the truncated system where $K$ and $\Theta$ are substituted by $K_M$ and $\Theta_M$ with $M\in\N$ as in \eqref{eqn:KTetha}.
		The whole calculations are carried out in the proof of Lemma \ref{lemma:aPriori}.
		At this point let us motivate in a more formal way where the Assumptions (A2) and (A3) originate:
		
		Testing \eqref{eqn:system1} with $-(w+1)^{-\alpha}$ (as done in the proof of Lemma \ref{lemma:aPriori}) yields after integration by parts
		\begin{align*}
			&\int_{\Omega_T}\Bigg(-w_t(w+1)^{-\alpha}-K_M(w)\nabla w\cdot\nabla\big((w+1)^{-\alpha}\big)\\
			&\qquad\quad+\big(\chi_t+\rho(\chi)\di\bl u_t\br+\rho'(\chi)\di(u)\chi_t\big)\frac{-\Theta_M(w)}{(w+1)^\alpha}\Bigg)\dxt=0.
		\end{align*}
		The second term may be rewritten as
		\begin{align*}
			-K_M(w)\nabla w\cdot\nabla\big((w+1)^{-\alpha}\big)
			=\frac{K_M(w)}{(w+1)^{\alpha+1}}|\nabla w|^2.
		\end{align*}
		On the one hand, to ensure the requested estimate
		$$
			\frac{K_M(w)}{(w+1)^{\alpha+1}}\geq c
		$$
		for some constant $c>0$ independent of $w$ and $M$, we assume $K(w)\geq c_1(w^{2q}+1)$ and $\alpha\leq 2q-1$.
		On the other hand, to guarantee boundedness of
		$$
			\frac{|\Theta_M(w)|}{(w+1)^\alpha}\leq C,
		$$
		we impose $\Theta(w)\leq c_0(w^{1/\sigma}+1)$ and $\alpha\geq\frac1\sigma$.
		Together, we obtain $\frac1\sigma\leq \alpha\leq 2q-1$.
		
		The choice of $\sigma$ can be determined by following the proof of the fifth a priori estimate in Lemma \ref{lemma:aPriori}.
		To this end, we test \eqref{eqn:system1} with $w$ and obtain
		\begin{align*}
			&\frac 12\int_\Omega|w(t)|^2\dx-\frac 12\int_\Omega|w(0)|^2\dx
			+\int_{\Omega_t} K_M(w)|\nabla w|^2\dxt\\
			&\quad+\int_{\Omega_t}\bl\chi_t+\rho(\chi)\di\bl u_t\br
				+\rho'(\chi)\di(u)\chi_t\br\Theta_M(w)w\dxt=0.
		\end{align*}
		From other a priori estimates it is known that the term
		$\chi_t+\rho(\chi)\di\bl u_t\br+\rho'(\chi)\di(u)\chi_t$ is bounded in $L^2(L^{3/2})$.
		Boundedness of $\Theta_M(w)w$ in $L^2(L^3)$ can be gained by the following splitting argument (this technique was already used in \cite{RR12})
		\begin{align*}
			\|\Theta_M(w)w\|_{L^2(0,t;L^3)}
			\leq{}& C\bl\int_0^t\|\Theta_M(w(s))w(s)\|_{L^3(\{w(s)\leq M\})}^2\ds\br^{1/2}\\
			&+C\bl\int_0^t\|\Theta_M(w(s))w(s)\|_{L^3(\{w(s)> M\})}^2\ds\br^{1/2}.
		\end{align*}
		It turns out that the first term on the right-hand side is bounded by the fourth a priori estimate.
		The second term can be estimated as follows (for more details we refer to the proof of the fifth a priori estimate in Lemma \ref{lemma:aPriori}):
		\begin{align*}
			\int_0^t\|\Theta_M(w(s))w(s)\|_{L^3(\{w(s)> M\})}^2\ds
			&=\esssup_{t\in(0,T)}\|\Theta_M(w(t))\|_{L^6(\{w(t)>M\})}^2\|w\|_{L^2(0,t;H^1(\Omega))}^2\\
			&\leq C\frac{M^{2/\sigma}+1}{M^{2/3}}\|w\|_{L^2(0,t;H^1(\Omega))}^2.
		\end{align*}
		To obtain an a priori bound uniformly in $M\uparrow\infty$, we impose $\sigma\geq 3$.
		
		The condition $q_0<q+\frac 12$ is necessary to ensure that
		the constant value $r=\frac{2q+2}{2q_0+1}$ is greater $1$.
		Because by a comparison argument performed in the sixth a priori estimate in Lemma \ref{lemma:aPriori}
		we can see that $w_t$ is bounded in the $W^{1,r}(0,T;W_\nu^{2,s}(\Omega)^*)$-space.
	\end{remark}
%	\begin{remark}
%		\textcolor{purple}{
%			We will use condition \eqref{eqn:viscousTensor} to obtain spatial $H^2$-regularity of $u$
%			in the time-discrete setting (see proof of Lemma \ref{lemma:discrSys}).
%		}
%	\end{remark}
	\begin{remark}
		As already indicated in the introduction, Assumption (A8) is used for mathematical reasons and plays a central role in handling the differential inclusion \eqref{eqn:system3}.
		More precisely, the compact embedding $W^{1,p}(\Omega)\hookrightarrow C^{0,\alpha}(\OL\Omega)$ with H\"older exponent $0<\alpha<1-\frac dp$
		is employed to apply approximation techniques introduced in \cite{WIAS1520} (see the item ``One-sided variational inequality for the damage process''
		in Subsection \ref{section:tauToZero})
		and enables us to pass from a time-discrete version of \eqref{eqn:system3} to a time-continuous limit.
	
		Secondly, Assumption (A8) is also utilized in the third a priori estimate in the proof of Lemma \ref{lemma:aPrioriDiscr}
		and Lemma \ref{lemma:aPriori} which are based on arguments in \cite[Proposition 3.10]{RR12}.
	\end{remark}

	For later use, we define the following subspaces (with $p,s\geq 1$):
	\begin{align*}
		&W_\nu^{2,s}(\Omega):=\big\{\zeta\in W^{2,s}(\Omega)\,|\,\nabla\zeta\cdot \nu=0\text{ on }\partial\Omega\big\},\\
		&W_+^{1,p}(\Omega):=\big\{\zeta\in W^{1,p}(\Omega)\,|\,\zeta\geq 0\text{ a.e. in }\Omega\big\},\\
		&W_-^{1,p}(\Omega):=\big\{\zeta\in W^{1,p}(\Omega)\,|\,\zeta\leq 0\text{ a.e. in }\Omega\big\}
	\end{align*}
	as well as the space-time cylinders $\Omega_T:=\Omega\times(0,T)$ and $\Omega_t:=\Omega\times(0,t)$.
	The primitive of an integrable function $f:\R\to\R$ vanishing at $0$ is denoted by $\widehat f$.

\section{Notion of weak solutions and main result}
\label{section:notion}
	
	We assume for the external heat source $g\in L^2(0,T;L^2(\Omega))$ and the external volume force $\ell\in L^2(0,T;L^2(\Omega;\R^d))$.
	We introduce the following notion of weak solutions (a justification of the weak notion for the $\chi$-equation is given in the subsequent Lemma \ref{lemma:justification}).

	\begin{definition}
	\label{def:weakSolution}
		A weak solution  corresponding to the initial data $(u^0,v^0,w^0,\chi^0)$ is a $4$-tuple $(u,w,\chi,\xi)$ such that
		\begin{align*}
			&u\in H^1(0,T;H_0^2(\Omega;\R^d))\cap W^{1,\infty}(0,T;H_0^1(\Omega;\R^d))\cap H^{2}(0,T;L^2(\Omega;\R^d))\\
			&\quad \text{ with }u(0)=u^0\text{ a.e. in }\Omega,\;\partial_t u(0)=v^0\text{ a.e. in }\Omega,\\
			&w\in L^2(0,T;H^1(\Omega))\cap L^{2(q+1)}(0,T;L^{6(q+1)}(\Omega))\cap L^\infty(0,T;L^2(\Omega))\\
			&\qquad\cap W^{1,r}(0,T;W_\nu^{2,s}(\Omega)^*)\\
			&\quad \text{ with }w(0)=w^0\text{ a.e. in }\Omega,\;w\geq 0\text{ a.e. in }\Omega_T,\\
			&\chi\in L^\infty(0,T;W^{1,p}(\Omega))\cap H^1(0,T;L^2(\Omega))\\
			&\quad \text{ with }\chi(0)=\chi^0\text{ a.e. in }\Omega,\;\chi\geq 0\text{ a.e. in }\Omega_T,\;\partial_t\chi\leq 0\text{ a.e. in }\Omega_T,\\
			&\xi\in L^1(0,T;L^1(\Omega))
		\end{align*}
		with $r:=(2q+2)/(2q_0+1)$ and $s:=(6q+6)/(6q-2q_0+5)$,
		and for a.e. $t\in(0,T)$:
		\begin{itemize}
			\item[(i)]
				heat equation: for all $\zeta\in W_\nu^{2,s}(\Omega)$
				\begin{align}
					&\langle\partial_t w,\zeta\rangle_{H^1}+\int_\Omega\bl-\widehat K(w)\Delta\zeta
						+\Theta(w)\partial_t \chi\zeta\br\dx\notag\\
					&\quad+\int_\Omega\big(\rho(\chi)\Theta(w)\di\bl \partial_t u\br\zeta
						+\rho'(\chi)\Theta(w)\di(u)\partial_t\chi\zeta\big)\dx=g,
				\label{eqn:weak1}
				\end{align}
			\item[(ii)]
				balance of forces: for a.e. $x\in\Omega$
				\begin{align}
					&\partial_{tt}u-\di\bl b(\chi)\CC\e(u)\br-\di\bl a(\chi)\DD\e(\partial_t u)\br+\di\bl\rho(\chi)\Theta(w)\mathbf 1\br=\ell,
				\label{eqn:weak2}
				\end{align}
			\item[(iii)]
				one-sided variational inequality: for all $\zeta\in W_-^{1,p}(\Omega)$
				\begin{align}
					&0\leq\int_\Omega\bl \partial_t \chi\zeta+|\nabla\chi|^{p-2}\nabla\chi\cdot\nabla\zeta+\gamma(\chi)\zeta
						+\frac{b'(\chi)}{2}\CC\e(u):\e(u)\zeta+\xi\zeta\br\dx\notag\\
				\label{eqn:weak3}
					&\qquad+\int_\Omega\bl-\Theta(w)\zeta-\rho'(\chi)\Theta(w)\di(u)\zeta\br\dx,
				\end{align}
				and $\xi\in\partial I_{W_+^{1,p}(\Omega)}(z)$, i.e.
				for all $\zeta\in W_+^{1,p}(\Omega)$
				\begin{align}
					&\int_\Omega\xi(\zeta-z)\dx\leq 0,
				\label{eqn:weak4}
				\end{align}
			\item[(iv)]
				partial energy inequality:
				\begin{align}
					&\int_\Omega\frac 1p|\nabla \chi(t)|^p\dx-\int_\Omega\frac 1p|\nabla \chi(s)|^p\dx
						+\int_{s}^t\int_\Omega\Big(\gamma(\chi)+\frac{b'(\chi)}{2}\CC\e(u):\e(u)\Big)\partial_t\chi\dx\diota\notag\\
					&\qquad+\int_{s}^t\int_\Omega\Big(-\Theta(w)-\rho'(\chi)\Theta(w)\di(u)+\partial_t\chi\Big)\partial_t\chi\dx\diota
					\leq 0
				\label{eqn:weak5}
				\end{align}
				for a.e. $0\leq s\leq t\leq T$ and for a.e. $t\in(0,T)$ with $s=0$
		\end{itemize}
		are satisfied.
	\end{definition}
	\begin{remark}
		Due to the assumption $q\leq q_0<q+\frac 12$ (see (A3)), it holds $1<r<2$.
	\end{remark}
	\begin{remark}
			Let us briefly describe how the precise values for the constants $r$ and $s$ in Definition \ref{def:weakSolution} arise.
			The application of a comparison argument in the proof of the sixth a priori estimate in Lemma \ref{lemma:aPriori} requires to bound the term
			\begin{align}
				\int_0^T\bl\sup_{\|\zeta\|_{W^{2,s}}=1}\int_\Omega \widehat K_M(w)\Delta\zeta\dx\br^r\dt
				&\leq\int_0^T\bl\sup_{\|\zeta\|_{W^{2,s}}=1}\|\widehat K_M(w)\|_{L^{s/(s-1)}}\|\Delta\zeta\|_{L^{s}}\br^r\dt\notag\\
			\label{eqn:KMboundExpl}
				&\leq C\int_0^T\|\widehat K_M(w)\|_{L^{s/(s-1)}}^r\dt,
			\end{align}
			where $\widehat K_M(x):=\int_0^x K_M(y)\dy$ is a primitive of $K_M(x)$.
			By boundedness of $w$ in $L^{2(q+1)}(L^{6(q+1)})$ and growth assumption for $K$ in (A3), we get boundedness
			of $\widehat K_M(w)$ in $L^{\frac{2q+2}{2q_0+1}}\big(L^{\frac{6q+6}{2q_0+1}}\big)$ (see \eqref{eqn:KMbound} for details).
			Hence to bound \eqref{eqn:KMboundExpl}, we need to choose $r$ and $s$ such that
			$$
				\frac{s}{s-1}=\frac{6q+6}{2q_0+1}\quad\text{and}\quad
				r=\frac{2q+2}{2q_0+1}.
			$$
	\end{remark}
	\begin{remark}
			Item (iv) in Definition \ref{def:weakSolution} is called ``partial energy inequality'' because it compares the
			potential $\frac 1p\|\nabla\chi(\cdot)\|_{L^p}^p$ as part of the free energy \eqref{eqn:freeEnergy} at two different time-points
			(namely $t$ and $s$).
	\end{remark}
	By assuming better regularity for $\chi$, it is seen from the one-sided variational inequality and the partial energy inequality
	that the desired differential inclusion \eqref{eqn:system3} holds in
	$W^{1,p}(\Omega)^*$.
	\begin{lemma}
	\label{lemma:justification}
		If a weak solution additionally fulfills $\chi\in H^1(0,T;W^{1,p}(\Omega))$ we obtain for a.e. $t\in(0,T)$
		\begin{align*}
			&\qquad\chi_t+\xi+\varphi-\Delta_p\chi+\gamma(\chi)+\frac{b'(\chi)}{2}\CC\e(u):\e(u)-\theta-\rho'(\chi)\theta\di(u)=0\text{ in }W^{1,p}(\Omega)^*
		\end{align*}
		with subgradients $\xi\in\partial I_{W_+^{1,p}(\Omega)}(\chi)$ and $\varphi\in\partial I_{W_-^{1,p}(\Omega)}(\chi_t)$.
		On the left-hand side the operator $\Delta_p:W^{1,p}(\Omega)\to W^{1,p}(\Omega)^*$ denotes the usual $p$-Laplacian with no-flux condition.
	\end{lemma}
	\begin{remark}
		We remark that the latter inclusion $\varphi\in\partial I_{W_-^{1,p}(\Omega)}(\chi_t)$ forces the monotonicity property $\chi_t\leq 0$ a.e. in
		$\Omega_T$.
	\end{remark}
	\begin{proof}
		By setting
		$$
			\varphi:=-\Big(\chi_t+\xi-\Delta_p\chi+\gamma(\chi)+\frac{b'(\chi)}{2}\CC\e(u):\e(u)-\theta-\rho'(\chi)\theta\di(u)\Big)\in W^{1,p}(\Omega)^*,
		$$
		and using (due to the enhanced regularity $\chi\in H^1(0,T;W^{1,p}(\Omega))$)
		$$
			\int_\Omega\frac 1p|\nabla \chi(t)|^p\dx-\int_\Omega\frac 1p|\nabla \chi^0|^p\dx
			=\int_0^t\int_\Omega |\nabla\chi|^{p-2}\nabla\chi\cdot\nabla\chi_t\dxs,
		$$
		property (iii) and property (iv) from Definition \ref{def:weakSolution} can be rewritten as
		\begin{align*}
			\big\langle\varphi\,,\;\zeta\big\rangle_{W^{1,p}}\leq 0\text{ and }
			-\big\langle\varphi\,,\;\chi_t\big\rangle_{W^{1,p}}\leq 0
		\end{align*}
		for all $\zeta\in W_-^{1,p}(\Omega)$ and a.e. $t\in(0,T)$.
		Here we have used the fact that $\left\langle \xi,\chi_t\right\rangle=0$.
		Adding these inequalities yields the inclusion $\varphi\in \partial I_{W_-^{1,p}(\Omega)}(\chi_t)$.
		\ep
	\end{proof}
	\begin{theorem}
	\label{theorem:existence}
		Let the Assumptions (A1)-(A8) be satisfied.
		Moreover, let the initial values $u^0\in H_0^2(\Omega;\R^d)$, $v^0\in H_0^1(\Omega;\R^d)$, $w^0\in L^2(\Omega)$ and $\chi^0\in W^{1,p}(\Omega)$ be given
		and assume that $w^0\geq 0$ and $0\leq\chi^0\leq 1$, $g\geq0$.
		Then, there exists a weak solution $(u,w,\chi,\xi)$ in the sense of Definition \ref{def:weakSolution}.
	\end{theorem}
	\begin{remark}
	\label{remark:homDirichlet}
			One essential ingredient in the proof of Theorem \ref{theorem:existence} is the application of the standard $H^2$-regularity result for elliptic systems
			using homogeneous Dirichlet boundary conditions on $u$ (see proof of Lemma \ref{lemma:discrSys}).
			Since more general regularity results are available,
			it seem conceivable to extend Theorem \ref{theorem:existence} to mixed-boundary conditions for $u$ where the
			Dirichlet and Neumann part
			%are relatively open in $\partial\Omega$ and
			satisfy $\OL{\Gamma_\mathrm{D}}\cap \OL{\Gamma_\mathrm{N}}=\emptyset$ and
			$\OL{\Gamma_\mathrm{D}}\cup\OL{\Gamma_\mathrm{N}}=\partial\Omega$. Indeed, without the latter  geometric condition,
			the elliptic regularity results ensuring the
			regularity of $u$ may fail to hold, see~\cite[Chap.~6,~Sec.~6.1]{MH94} and \cite[Chap.~VI,~Sec.~6.3]{Cia88}.
		
			However, since the main focus of this work is laid on the analysis of coupling different physical processes,
			we decided to restrict ourselves to case of homogeneous Dirichlet boundaries for the displacement field.
	\end{remark}
	The proof is carried out in the following two sections. It is based on a time-discretization scheme and on an approximation argument involving a truncation of
	$K$ and $\Theta$ (cf. also \cite{RR12}).

\section{Existence of weak solutions for the truncated system}
 \label{section:existence1}
	To keep the presentation short, we assume $g=\ell=0$ in \eqref{eqn:system1}-\eqref{eqn:system2}.
%	\textcolor{red}{as well as $\CC=\DD=\mathds 1$.}
	
		The existence proof presented in this section is based on a double approximation technique: a truncation and a time-discretization scheme.
		Let us point out why truncation on the coefficients $K$ and $\Theta$ is used in the first place:
	
		Assume for a moment that we have established the energy estimate (first a priori estimate; see Lemma \ref{lemma:aPrioriDiscr}).
		Then, to obtain an $L^2(0,T;H^1(\Omega))\cap L^{2(q+1)}(0,T;L^{6(q+1)}(\Omega))\cap L^\infty(0,T;L^2(\Omega))$-estimate
		for the enthalpy $w$, it is necessary to test the enthalpy equation \eqref{eqn:system1} with $w$ (fourth a priori estimate).
		This results to further challenges since one needs to guarantee that the integral $\int_{\Omega_T}K(w)|\nabla w|^2$ exists.
		Beyond that, we need an $L^\infty(L^6)$-bound available for $u$ (in 3D).
		To this end, the equation \eqref{eqn:system2} should be tested with $-\di(\e(u_t))$ (third a priori estimate).
		But because of the $\int_{\Omega_T}\di(\rho(\chi)\Theta(w)\mathds 1)\cdot\di(\e(u_t))$-term,
		an $L^2(L^2)$-bound for $\nabla(\Theta(w))$ is necessary.
		This advises us to test the enthalpy equation with $\Theta(w)$ (second a priori estimate) in order to obtain the mentioned $L^2(L^2)$-bound.
		We are then faced with the difficulty to estimate $\Theta(w)$ in $L^\infty(L^\infty)$.
	
		To overcome these difficulties, 
	we firstly prove existence of weak solutions
	to a truncated system of \eqref{eqn:system1}-\eqref{eqn:system3} where $K$ and $\Theta$ are substituted by $K_M$ and $\Theta_M$ for $M\geq 0$ defined by
	%	Furthermore, we will make use of the truncated functions $\Theta_M$ of $\Theta$ at $M\geq 0$, i.e.
	\begin{align}
	\label{eqn:KTetha}
		\Theta_M(x):=
		\begin{cases}
			\Theta(M)&\text{if }x>M,\\
			\Theta(x)&\text{if }-M\leq x\leq M,\\
			\Theta(-M)&\text{if }x<-M,
		\end{cases}
		\qquad K_M(x):=
		\begin{cases}
			K(M)&\text{if }x>M,\\
			K(x)&\text{if }-M\leq x\leq M,\\
			K(-M)&\text{if }x<-M.
		\end{cases}
	\end{align}
	We remind that $\Theta(w)=0$ for all $w\leq 0$ by Assumption (A3).

	The truncation function $\C T_M:\R\to\R$ at the height $M$ is given via
	\begin{align*}
		\C T_M(x)=
		\begin{cases}
			M&\text{if }x>M,\\
			x&\text{if }-M\leq x\leq M,\\
			-M&\text{if }x<-M.
		\end{cases}
	\end{align*}
	Note that the crucial properties $\Theta_M(w)=\Theta(\C T_M(w))$ and $K_M(w)=K(\C T_M(w))$ are satisfied.
	
		In order to prove Theorem \ref{theorem:existence}, the investigation of the limit $M\uparrow\infty$
		requires more elaborate a priori estimates which are postponed to Section \ref{section:existence2}.
		Let us remark that the estimates there can also be adapted to a time-discrete version
		such that the limits $\tau\downarrow 0$ and $M\uparrow\infty$ may be performed simultaneously.
		However, we decided to separate this passage into two steps and show that the limit analysis for $\tau\downarrow 0$
		and fixed $M\in\N$ are conducted with fairly simpler arguments.

\subsection{Time-discrete system}
\label{section:timediscrSys}	
	In this subsection, we will prove existence of weak solutions for a time-discrete and truncated
	version of system \eqref{eqn:system1}-\eqref{eqn:system3} by using a semi-implicit Euler scheme.
	The scheme is carefully chosen such that we can derive an energy estimate (see Lemma \ref{lemma:aPrioriDiscr} (i)).
	
	To this end, we consider an equidistant partition
	$\{0,\tau,2\tau,\ldots,T\}$ of $[0,T]$ where $\tau>0$ denotes the time-discretization fineness.
	Moreover, let $T_\tau:=T/\tau$ be the final time index (note that $T/\tau\in\N$ by the assumed equidistancy of the partition).
	We set $(u_\tau^0,w_\tau^0,\chi_\tau^0):=(u^0,w^0,\chi^0)$ and  $u_\tau^{-1}:=u^0-\tau v^0$
	and perform a recursive procedure.
	
	In the following, we adopt the notation $D_{\tau,k}(w)=\tau^{-1}(w_\tau^k-w_\tau^{k-1})$ (as well as for $D_{\tau,k}(u)$ and $D_{\tau,k}(\chi)$).
	Let, furthermore, $v_\tau^k$ be defined as
	\begin{align}
	\label{eqn:vDef}
		v_\tau^k:=\frac{u_\tau^k-u_\tau^{k-1}}{\tau}.
	\end{align}
	
	The existence of weak solutions for the time-discrete system is proven in the following.
	\begin{lemma}
	\label{lemma:discrSys}
		For every equidistant partition of $[0,T]$ with fineness $\tau>0$, there exists a sequence $\{(u_\tau^k,w_\tau^k,\chi_\tau^k,\xi_\tau^k)\}_{k=1}^{T_\tau}$ in the space
		$H_0^2(\Omega;\R^d)\times H^1(\Omega)\times W^{1,p}(\Omega)\times W^{1,p}(\Omega)^*$
		such that for all $k\in\{1,\ldots,T_\tau\}$:
		\begin{enumerate}
			\item[(i)]
				for all $\zeta\in H^{1}(\Omega)$
				\begin{align}
					&\int_\Omega\bl D_{\tau,k}(w)\zeta+K_{M}(w_\tau^{k-1})\nabla w_\tau^{k}\cdot\nabla\zeta
						+\Theta_M(w_\tau^{k-1})D_{\tau,k}(\chi)\zeta\br\dx\notag\\
					&\quad+\int_\Omega\rho(\chi_\tau^{k-1})\Theta_M(w_\tau^{k})\di\bl D_{\tau,k}(u)\br\zeta\dx\notag\\
				\label{eqn:discrHeatEq}
					&\quad+\int_\Omega\rho'(\chi_\tau^{k-1})\Theta_M(w_\tau^{k})\di(u_\tau^{k-1})D_{\tau,k}(\chi)\zeta\dx=0,
				\end{align}
			\item[(ii)]
				for a.e. $x\in\Omega$
				\begin{align}
					&D_{\tau,k}^2(u)-\di\bl b(\chi_\tau^k)\CC\e(u_\tau^k)\br
						-\di\bl a(\chi_\tau^k)\DD\e(D_{\tau,k}(u_\tau^k))\br\notag\\
				\label{eqn:discrMomentumEq}
					&\qquad+\di\bl\rho(\chi_\tau^{k-1})\Theta_M(w_\tau^{k})\mathbf 1\br=0,
				\end{align}
			\item[(iii)]
				for all $\zeta\in W^{1,p}(\Omega)$%with $0\leq\zeta+\chi_\tau^k\leq \chi_\tau^{k-1}$
				\begin{align}
					&0=\int_\Omega\bl D_{\tau,k}(\chi)\zeta+|\nabla\chi_\tau^k|^{p-2}\nabla\chi_\tau^k\cdot\nabla\zeta+\gamma(\chi_\tau^k)\zeta
						-\Theta_M(w_\tau^{k-1})\zeta\br\dx\notag\\
				\label{eqn:discrDamageEq}
					&\qquad+\int_\Omega\bl\frac{b_1'(\chi_\tau^{k})+b_2'(\chi_\tau^{k-1})}{2}\CC\e(u_\tau^{k-1}):\e(u_\tau^{k-1})\zeta
						-\rho'(\chi_\tau^{k-1})\Theta_M(w_\tau^{k})\di(u_\tau^{k-1})\zeta\br\dx\notag\\
					&\qquad+\langle\xi,\zeta\rangle_{W^{1,p}}
				\end{align}
				with $\xi\in \partial I_{Z_\tau^{k-1}}(\chi_\tau^k)$, where $Z_\tau^{k-1}$ is given by
				\begin{align*}
					Z_\tau^{k-1}:=\big\{f\in W^{1,p}(\Omega)\,|\,0\leq f\leq \chi_\tau^{k-1}\big\}
				\end{align*}
				and $b=b_1+b_2$ denotes a convex-concave decomposition of $b$, e.g.
				\begin{align*}
					&b_1(r):= b(0)+\int_0^r\Big(b'(0)+\int_0^s \max\{b''(\mu),0\}\,\mathrm d\mu\Big)\ds,\\
					&b_2(r):= \int_0^r\Big(\int_0^s \min\{b''(\mu),0\}\,\mathrm d\mu\Big)\ds.
				\end{align*}
		\end{enumerate}
	\end{lemma}
	\begin{remark}
			The combination of explicit and implicit terms in the time-discretization is chosen
			in a way that the energy estimate
			is obtained by testing \eqref{eqn:discrHeatEq} with $\tau$,
			testing \eqref{eqn:discrMomentumEq} with $u_\tau^k-u_\tau^{k-1}$,
			testing \eqref{eqn:discrDamageEq} with $\chi_\tau^k-\chi_\tau^{k-1}$
			and adding them together (see the proof of the first a priori estimate in Lemma \ref{lemma:aPrioriDiscr}).
			In particular, the convex-concave decomposition of $b$ ensures the crucial estimate \eqref{eqn:bConvConc}.
			Then the terms
			\begin{align*}
				&\frac12(b_1'(\chi_\tau^{k})+b_2'(\chi_\tau^{k-1}))(\chi_\tau^{k-1}-\chi_\tau^k)\CC\e(u_\tau^{k-1}):\e(u_\tau^{k-1}),
				&&\Theta_M(w_\tau^{k-1})D_{\tau,k}(\chi),\\
				&\rho(\chi_\tau^{k-1})\Theta_M(w_\tau^{k})\di\bl D_{\tau,k}(u)\br,
				&&\rho'(\chi_\tau^{k-1})\Theta_M(w_\tau^{k})\di(u_\tau^{k-1})D_{\tau,k}(\chi)
			\end{align*}
			cancel out.
	\end{remark}
	\begin{proof}[Proof of Lemma \ref{lemma:discrSys}]
		We will trace back this PDE problem to the abstract inclusion problem
		\begin{align}
		\label{eqn:abstractInclusion}
			\partial\Psi(p)+A(p)\ni f,
		\end{align}
		where $A:X\to X^*$ is pseudomonotone and $\Psi:X\to\R\cup\{+\infty\}$ is a convex, proper and lower semicontinuous functional.
		
			Existence of solutions for \eqref{eqn:abstractInclusion} is ensured in \cite[Theorem 5.15]{Rou13} via a Leray-Lions type theorem for non-potential inclusions if $\Psi$
			possess
		a convex and G\^{a}teaux differentiable regularization $\Psi_\varepsilon:X\to\R$ such that $\Psi_\varepsilon$
		is bounded and radially continuous and
		\begin{align*}
			&\limsup_{\varepsilon\downarrow 0}\Psi_\varepsilon(g)\leq \Psi(g)\text{ for all }g\in X,\\
			&\liminf_{\varepsilon\downarrow 0}\Psi_\varepsilon(g_\varepsilon)\geq\Psi(g)\text{ for all }g_\varepsilon\to g\text{ weakly in } X.
		\end{align*}
		
		In our case, we use the spaces
		\begin{align*}
			&X=H_0^1(\Omega;\R^d)\times H^1(\Omega)\times W^{1,p}(\Omega),\\
			&Y=\big\{(u,w,\chi)\in H_0^1(\Omega;\R^d)\times H^1(\Omega)\times W^{1,p}(\Omega)\,|\,0\leq \chi\leq \chi_\tau^{k-1}\big\}\subseteq X
		\end{align*}
		and the operators (we write $A=(A_1,A_2,A_3)$)
		\begin{align*}
			&\Psi=I_{Y}\text{ (indicator function $I_Y:X\to\R\cup\{\infty\}$ of the set $Y$) },\\
			&A_1(u,w,\chi)=u-\tau^2\di\bl b(\chi)\CC\e(u)\br-\tau\di\bl a(\chi)\DD\e(u-u_\tau^{k-1})\br\\
			&\hspace*{5.9em}+\tau^2\di\bl\rho(\chi_\tau^{k-1})\Theta_M(w)\mathbf 1\br,\\
			&A_2(u,w,\chi)=w-\tau\di\bl K_M(w_\tau^{k-1})\nabla w\br+\Theta_M(w_\tau^{k-1})\chi+\rho(\chi_\tau^{k-1})\Theta_M(w)\di\bl u-u_\tau^{k-1}\br\\
			&\hspace*{5.9em}+\rho'(\chi_\tau^{k-1})\Theta_M(w)\di\bl u_\tau^{k-1}\br(\chi-\chi_\tau^{k-1}),\\
			&A_3(u,w,\chi)=\chi-\tau\Delta_p\chi+\frac{b_1'(\chi)}{2}\CC\e(u_\tau^{k-1}):\e(u_\tau^{k-1})+\gamma(\chi)-\rho'(\chi_\tau^{k-1})\Theta_M(w)\di(u_\tau^{k-1})
		\end{align*}
		and the element $f\in X^*$ given by
		\begin{align*}
			f=
			\left(
			\begin{matrix}
				2u_\tau^{k-1}-u_\tau^{k-2}\hspace*{10.8em}\\
				w_\tau^{k-1}+\Theta_M(w_\tau^{k-1})\chi_\tau^{k-1}\hspace*{6.6em}\\
				\chi_\tau^{k-1}-\frac{b_2'(\chi_\tau^{k-1})}{2}|\e(u_\tau^{k-1})|^2+\Theta_M(w_\tau^{k-1})
			\end{matrix}
			\right).
		\end{align*}
		Note that $Y$ is a convex, nonempty and closed subspace of $X$ since $\chi_\tau^{k-1}\in C(\OL{\Omega})$.
		
		Now, it can be checked that the operator $A$ is pseudomonotone and coercive.
			The regularization $\Psi_\varepsilon$ can be chosen to be
			the Yosida approximation $\Psi_\varepsilon$ of $\Psi$.
			Finally, the existence result in \cite[Theorem 5.15]{Rou13}
			yields a solution to the problem \eqref{eqn:abstractInclusion} and, therefore, to (i)-(iii).
		
			It remains to show $H_0^2(\Omega;\R^d)$-regularity of $u_\tau^k$.
			Let us consider the three-dimensional case $d=3$.
%			To emphasize where condition \eqref{eqn:viscousTensor} is needed, we involve the tensors $\CC$ and $\DD$
%			in the calculations.
			
			First of all, we rewrite \eqref{eqn:discrMomentumEq} in the following form
			\begin{align*}
				\int_\Omega \big(\tau^2 b(\chi_\tau^k)\CC+\tau a(\chi_\tau^k)\DD\big)\e(u_\tau^k):\e(\zeta)\dx
				=\int_\Omega g\cdot\zeta\dx
			\end{align*}
			valid for all $\zeta\in H_0^1(\Omega;\R^d)$ and with right-hand side
			(note that $u_\tau^{k-1}\in H_0^2(\Omega;\R^d)$)
			$$
				g:=-u_\tau^k+2u_\tau^{k-1}-u_\tau^{k-2}-\tau\di\bl a(\chi_\tau^k)\DD\e(u_\tau^{k-1})\br-\tau^2\di\bl\rho(\chi_\tau^{k-1})\Theta_M(w_\tau^k)\mathbf 1\br
				\in L^2(\Omega;\R^d).
			$$
			Condition \eqref{eqn:viscousTensor} shows
			\begin{align}
			\label{eqn:ellipticEq}
				\int_\Omega \big(\tau^2 b(\chi_\tau^k)+\tau a(\chi_\tau^k)\mu\big)\CC\e(u_\tau^k):\e(\zeta)\dx
				=\int_\Omega g\cdot\zeta\dx.
			\end{align}
			Since the coefficient function $\tau^2 b(\chi_\tau^k)+\tau a(\chi_\tau^k)\mu\in W^{1,p}(\Omega)$ in \eqref{eqn:ellipticEq}
			is scalar-valued and bounded from below by a positive constant (see (A5)), we get
			$
				\bl\tau^2 b(\chi_\tau^k)+\tau a(\chi_\tau^k)\mu\br^{-1}\in W^{1,p}(\Omega).
			$
			Testing \eqref{eqn:ellipticEq} with $\zeta=\bl\tau^2 b(\chi_\tau^k)+\tau a(\chi_\tau^k)\mu\br^{-1}\varphi$ where $\varphi\in H_0^1(\Omega;\R^d)$ is another test-function yields
			\begin{align}
			\label{eqn:ellipticEqTrans}
				\int_\Omega \CC\e(u_\tau^k):\e(\varphi)\dx
				=\int_\Omega \widehat g\cdot\varphi\dx
			\end{align}
			with the new right-hand side
			\begin{align}
			\label{eqn:hatG}
				\widehat g:=\frac{1}{\tau^2 b(\chi_\tau^k)+\tau a(\chi_\tau^k)\mu}g
					+\CC\e(u_\tau^k)\cdot\frac{\tau^2 b'(\chi_\tau^k)+\tau a'(\chi_\tau^k)\mu}{\tau^2 b(\chi_\tau^k)+\tau a(\chi_\tau^k)\mu}\nabla\chi_\tau^k.
			\end{align}
			Since $\nabla\chi_\tau^k\in L^p(\Omega;\R^d)$ and $\e(u_\tau^k)\in L^2(\Omega;\R^{d\times d})$, we get
			$\widehat g\in L^{2p/(2+p)}(\Omega;\R^d)$.
%			By using $p>d$, we know $\widehat g\in L^{\frac{6}{5}+\varepsilon}(\Omega;\R^d)$ for some $\varepsilon>0$.
			
			Now, we will use an iteration argument in combination with a regularity result from \cite{MH94}
%			(see \cite{MN10,CDN10} for anisotropic stiffness tensors as in \eqref{eqn:anisotropy} from Assumption (A6))
			applied to the linear elasticity system \eqref{eqn:ellipticEqTrans}
			to gain $H^2$-regularity for $u_\tau^k$.
			
			To this end, let us assume that $\widehat g$ is in some $L^s(\Omega;\R^{d})$-space with $s\in [6/5,2]$.
			At the beginning this will be $s=s_0:=2p/(2+p)$ (see above) which is greater than $6/5$ due to $p>3$ by (A8).
			By using Assumption (A6), the regularity theorem in
			\cite[Chap.~6,~Theorem~1.11~(i)]{MH94} (see also \cite[Theorem 6.3-6]{Cia88} for isotropic $\CC$)
			shows that $u_\tau^k\in W^{2,s}(\Omega;\R^{d})$.
			By using the Sobolev embedding theorem, we obtain $\e(u_\tau^k)\in L^{3s/(3-s)}(\Omega;\R^{d\times d})$.
			This, on the other hand, implies
			\begin{align}
			\label{eqn:gIntegrability}
				\widehat g\in L^{\min\big\{\frac{3ps}{3p +3s - ps},2\big\}}(\Omega;\R^{d}).
			\end{align}
			Therefore, after applying the $W^{2,s}$-regularity result, we obtain enhanced integrability of the right-hand side $\widehat g$.
			To see that $\widehat g\in L^2(\Omega;\R^d)$ can be obtained after finitely many iterations, we consider the function (which occurs in \eqref{eqn:gIntegrability})
			$$
				f(s):=\frac{3ps}{3p+3s-ps}
			$$
			and see that $[f(s)-s]'\geq 0$ for all $s\in[s_0,2]$. Thus the increase of integrability before reaching the value $2$ can be bounded from below by a positive constant
			(provided that $p>3$):
			$$
				f(s)-s=\frac{(p-3)s^2}{3p+3s-ps}\geq f(s_0)-s_0=\frac{(p-3)s_0^2}{3p+3s_0-ps_0}>0
			$$
			for all $s\in[s_0,2]$.
			%$f'(s),f''(s)>0$ for all $s\in(0,\frac{3p}{p-3})$ and $f(s)\to+\infty$ as $s\uparrow\frac{3p}{p-3}$.
			Once we have obtained $\widehat g\in L^2(\Omega;\R^d)$, \cite[Theorem 1.11 (i)]{MH94} yields $u_\tau^k\in H_0^2(\Omega;\R^d)$ as desired.
		\ep
	\end{proof}

	For later use, we define for a sequence of functions $\{h_\tau^k\}_{0\leq k\leq T_\tau}$ the piecewise constant and linear interpolation
	on the time interval $(0,T)$ as
	\begin{align*}
		\OL{h}_\tau(t):=h_\tau^k,\qquad \UL{h}_\tau(t):=h_\tau^{k-1},\qquad h_\tau:=\frac{t-(k-1)\tau}{\tau}h_\tau^k+\frac{k\tau-t}{\tau}h_\tau^{k-1}
	\end{align*}
	for $t\in((k-1)\tau,k\tau]$.
	Given a $t\in[0,T]$, we denote by $\OL t_\tau$ and $\UL t_\tau$ the left- and right-continuous piecewise constant interpolation, i.e.
	\begin{align*}
		&\OL t_\tau:=\tau k\text{ for }\tau(k-1)<t\leq\tau k,\\
		&\UL t_\tau:=\tau (k-1)\text{ for }\tau(k-1)\leq t<\tau k.
	\end{align*}
	
	In what follows, we take for every $\tau>0$ a time-discrete weak solution in the sense of Lemma \ref{lemma:discrSys}
	and adopt the convention above.
	\begin{remark}
		The differential inclusion \eqref{eqn:discrDamageEq} is equivalent to the following variational inequality:
		\begin{align}
			&0\geq-\int_\Omega|\nabla\overline{\chi}_\tau|^{p-2}\nabla\overline{\chi}_\tau\cdot\nabla(\zeta-\overline{\chi}_\tau)\dx\\
			&\qquad-\int_\Omega\bl\partial_t \chi_\tau
				+\gamma(\overline{\chi}_\tau)
				+\frac 12 b'(\overline{\chi}_\tau)\CC\e(\underline{u}_\tau):\e(\underline{u}_\tau)\br(\zeta-\overline{\chi}_\tau)\dx\notag\\
			&\qquad-\int_\Omega\bl-\Theta_M(\underline{w}_\tau)
				-\rho'(\underline{\chi}_\tau)\Theta_M(\overline{w}_\tau)\di(\underline{u}_\tau)\br(\zeta-\overline{\chi}_\tau)\dx
			\label{eqn:damageVIdiscr}
		\end{align}
		holding for all $\zeta\in W^{1,p}(\Omega)$ with $0\leq \zeta\leq \underline{\chi}_\tau$.
	\end{remark}

\subsection{A priori estimates}
	We are going to prove a~priori estimates for the discrete system in Lemma \ref{lemma:discrSys}.
	We will make use of the following implication.
	\begin{lemma}
	\label{label:positivity}
		A time-discrete weak solution constructed in the previous subsection satisfies
		$\overline w_M\geq 0$ and $\underline w_M\geq 0$.
	\end{lemma}
	\begin{proof}
		We show this lemma by induction over $k\in\{0,\ldots,T_\tau\}$. Assume that $w_\tau^{k-1}$ fulfills $w_\tau^{k-1}\geq 0$.
		Testing equation \eqref{eqn:discrHeatEq} with $\zeta=-(w_\tau^k)^-:=\min\{w_\tau^k,0\}$ yields
		\begin{align*}
			&\frac 1\tau\int_\Omega\underbrace{-w_\tau^k (w_\tau^k)^-}_{=|(w_\tau^k)^-|^2}\dx+\frac 1\tau\int_\Omega\underbrace{w_\tau^{k-1}(w_\tau^k)^-}_{\geq 0}\dx
				+\int_\Omega \underbrace{K_M(w_\tau^{k-1})\nabla w_\tau^{k}\cdot\nabla (-(w_\tau^k)^-)}_{\geq c_1|\nabla (w_\tau^k)^-|^2\text{ by (A3)}}\dx\\
			&\quad+\int_\Omega\underbrace{\Theta_M(w_\tau^{k-1})}_{\geq 0}\underbrace{D_{\tau,k}(\chi)}_{\leq 0}\underbrace{(-(w_\tau^k)^-)}_{\leq 0}\dx
				+\int_\Omega\rho(\chi_\tau^{k-1})\di\bl D_{\tau,k}(u)\br\underbrace{\Theta_M(w_\tau^{k})(-(w_\tau^k)^-)}_{=0}\dx\\
			&\quad+\int_\Omega\rho'(\chi_\tau^{k-1})\di(u_\tau^{k-1})D_{\tau,k}(\chi)\underbrace{\Theta_M(w_\tau^{k})(-(w_\tau^k)^-)}_{=0}\dx=0.
		\end{align*}
		\ep
	\end{proof}
	\begin{remark}
			The technique used in the proof of Lemma \ref{label:positivity} is not applicable to show $w_\tau^k\geq w^0\geq 0$ provided that
			$g\geq 0$.
			Because, by testing \eqref{eqn:discrHeatEq} with $\zeta=-(w_\tau^k-w^0)^-$ (see \cite[Lemma 3.8]{RR12}),
			we cannot conclude that $\Theta_M(w_\tau^{k})(w_\tau^k-w^0)^-$ is $0$. On the other hand, the technique recently used in \cite{RR14} cannot be applied here because it is strictly related to the presence of the dissipative quadratic terms \eqref{diss} on the right hand side of \eqref{eqn:pde1}.
	\end{remark}

	\begin{lemma}[A~priori estimates independent of $\tau$]
	\label{lemma:aPrioriDiscr}
	The following a~priori estimates hold:
	\begin{align*}
		\textit{(i)}&\textit{ First a priori estimate (uniformly in $\tau$ and $M$):}\hspace*{-9em}\notag\\
		&\quad\{u_\tau\}&&\text{ in }H^1(0,T;H^1(\Omega;\R^d))\cap W^{1,\infty}(0,T;L^2(\Omega;\R^d)),\\
		&\quad\{\overline u_\tau\},\{\underline u_\tau\}&&\text{ in }L^\infty(0,T;H^1(\Omega;\R^d)),\\
		&\quad\{\overline{w}_\tau\},\{\underline{w}_\tau\}&&\text{ in }L^\infty(0,T;L^1(\Omega)),\\
		&\quad\{\chi_\tau\}&&\text{ in }L^\infty(0,T;W^{1,p}(\Omega))\cap H^1(0,T;L^2(\Omega)),\\
		&\quad\{\overline\chi_\tau\},\{\underline\chi_\tau\}&&\text{ in }L^\infty(0,T;W^{1,p}(\Omega)),\\
		\textit{(ii)}&\textit{ Second a priori estimate (uniformly in $\tau$):}\hspace*{-9em}\notag\\
		&\quad\{\nabla\Theta_M(\overline w_\tau)\}&&\text{ in }L^2(0,T;L^2(\Omega)),\\
		\textit{(iii)}&\textit{ Third a priori estimate (uniformly in $\tau$):}\hspace*{-9em}\notag\\
		&\quad\{u_\tau\}&&\text{ in }H^1(0,T;H^{2}(\Omega;\R^d))\cap W^{1,\infty}(0,T;H^1(\Omega;\R^d)),\\
		&\quad\{\overline{u}_\tau\},\{\underline{u}_\tau\}&&\text{ in }L^\infty(0,T;H^{2}(\Omega;\R^d)),\\
		&\quad\{v_\tau\}&&\text{ in }L^2(0,T;H^2(\Omega;\R^d))\cap L^{\infty}(0,T;H^1(\Omega;\R^d))\notag\\
		&&&\quad\;\cap H^1(0,T;L^{2}(\Omega;\R^d)),\\
		\textit{(iv)}&\textit{ Fourth a priori estimate (uniformly in $\tau$):}\hspace*{-9em}\notag\\
		&\quad\{\overline w_\tau\},\{\underline w_\tau\}&&\text{ in }L^2(0,T;H^1(\Omega))\cap L^{\infty}(0,T;L^2(\Omega)),\\
		\textit{(v)}&\textit{ Fifth a priori estimate (uniformly in $\tau$):}\hspace*{-9em}\notag\\
		&\quad\{w_\tau\}&&\text{ in }H^1(0,T;H^1(\Omega)^*).
	\end{align*}
	\end{lemma}
%	\begin{lemma}[First a priori estimate]
%	\label{lemma:firstAPE}
%		It holds $w_\tau^k\geq 0$ a.e. in $\Omega$ and
%		\begin{align*}
%			&\|w_\tau^k\|_{L^1(\Omega)}\leq C,
%			&&\|D_{\tau,k}(u)\|_{L^2(\Omega)}\leq C,\\
%			&\|\e(u_\tau^k)\|_{L^2(\Omega)}\leq C,
%			&&\|\chi_\tau^k\|_{W^{1,p}(\Omega)}\leq C,\\
%			&\sum_{l=0}^k\tau\|D_{\tau,l}(\e(u))\|_{L^2(\Omega)}\leq C,
%			&&\sum_{l=0}^k\tau\|D_{\tau,l}(\chi)\|_{L^2(\Omega)}\leq C
%		\end{align*}
%		for a constant $C>0$ independent of $\tau$ and $k=0,\ldots,T_\tau$.
%	\end{lemma}
	\begin{proof}[Proof of the first a~priori estimate]
		The first a~priori estimate is based on adding equation \eqref{eqn:system1} tested by $1$
		with equation \eqref{eqn:system2} tested by $\partial_t u$ and with equation \eqref{eqn:system3} tested by $\partial_t \chi$.
		Here, we will develop this estimate on a time-discrete level.
		
		In the following, we make use of a convex-concave estimate for $b_1$ and $b_2$ given by
		\begin{align*}
			b(\chi_\tau^{k-1})-b(\chi_\tau^k)
			&=\big(b_1(\chi_\tau^{k-1})-b_1(\chi_\tau^k)\big)
			+\big(b_2(\chi_\tau^{k-1})-b_2(\chi_\tau^k)\big)\\
			&\geq b_1'(\chi_\tau^{k})(\chi_\tau^{k-1}-\chi_\tau^k)
				+b_2'(\chi_\tau^{k-1})(\chi_\tau^{k-1}-\chi_\tau^k)\\
			&=(b_1'(\chi_\tau^{k})+b_2'(\chi_\tau^{k-1}))(\chi_\tau^{k-1}-\chi_\tau^k).
		\end{align*}
		Testing \eqref{eqn:discrMomentumEq} with $\zeta=u_\tau^k-u_\tau^{k-1}$,
		using the combined convex-concave estimate (note the positivity of $\CC$)
		\begin{align}
			&b(\chi_\tau^k)\CC\e(u_\tau^k):\e(u_\tau^k-u_\tau^{k-1})\notag\\
			&\quad= \frac{b(\chi_\tau^k)}{2}\CC\e(u_\tau^k):\e(u_\tau^k)
				-\frac{b(\chi_\tau^{k-1})}{2}\CC\e(u_\tau^{k-1}):\e(u_\tau^{k-1})\notag\\
			&\qquad+\frac12(b(\chi_\tau^{k-1})-b(\chi_\tau^k))\CC\e(u_\tau^{k-1}):\e(u_\tau^{k-1})
				+\frac{b(\chi_\tau^k)}{2}\CC\e(u_\tau^k-u_\tau^{k-1}):\e(u_\tau^k-u_\tau^{k-1})\notag\\
			&\quad\geq \frac{b(\chi_\tau^k)}{2}\CC\e(u_\tau^k):\e(u_\tau^k)
				-\frac{b(\chi_\tau^{k-1})}{2}\CC\e(u_\tau^{k-1}):\e(u_\tau^{k-1})\notag\\
			&\qquad+\frac12(b(\chi_\tau^{k-1})-b(\chi_\tau^k))\CC\e(u_\tau^{k-1}):\e(u_\tau^{k-1})\notag\\
			&\quad\geq\frac{b(\chi_\tau^k)}{2}\CC\e(u_\tau^k):\e(u_\tau^k)
			-\frac{b(\chi_\tau^{k-1})}{2}\CC\e(u_\tau^{k-1}):\e(u_\tau^{k-1})\notag\\
		\label{eqn:bConvConc}
			&\qquad+\frac12(b_1'(\chi_\tau^{k})+b_2'(\chi_\tau^{k-1}))(\chi_\tau^{k-1}-\chi_\tau^k)\CC\e(u_\tau^{k-1}):\e(u_\tau^{k-1})
		\end{align}
		and
		\begin{align*}
			D_{\tau,k}^2(u)\cdot(u_\tau^k-u_\tau^{k-1})\geq \frac12 |D_{\tau,k}(u)|^2-\frac12 |D_{\tau,k-1}(u)|^2,
		\end{align*}
		yield
		\begin{align}
			&\frac 12\|D_{\tau,k}(u)\|_{L^2}^2-\frac 12\|D_{\tau,k-1}(u)\|_{L^2}^2
				+\int_\Omega\frac{b(\chi_\tau^k)}{2}\CC\e(u_\tau^k):\e(u_\tau^k)\dx\notag\\
		\label{eqn:eq1}
 			&\qquad-\int_\Omega\frac{b(\chi_\tau^{k-1})}{2}\CC\e(u_\tau^{k-1}):\e(u_\tau^{k-1})\dx
 				+\tau\int_\Omega a(\chi_\tau^k)\DD\e(D_{\tau,k}(u)):\e(D_{\tau,k}(u))+R_1\leq 0
		\end{align}
		with the remainder term
		\begin{align*}
			R_1:={}&\int_\Omega\frac{b_1'(\chi_\tau^{k})
				+b_2'(\chi_\tau^{k-1})}{2}\CC\e(u_\tau^{k-1}):\e(u_\tau^{k-1})(\chi_\tau^{k-1}-\chi_\tau^k)\dx\\
			&-\int_\Omega\rho(\chi_\tau^{k-1})\Theta_M(w_\tau^{k})\di\bl u_\tau^k-u_\tau^{k-1}\br\dx.
		\end{align*}
		Testing \eqref{eqn:discrDamageEq} with $\chi_\tau^{k-1}-\chi_\tau^{k}$ and using the convexity estimate
		\begin{align*}
			&|\nabla\chi_\tau^k|^{p-2}\nabla\chi_\tau^k\cdot\nabla(\chi_\tau^k-\chi_\tau^{k-1})
				\geq \int_\Omega\frac1p|\nabla\chi_\tau^k|^p\dx-\int_\Omega\frac1p|\nabla\chi_\tau^{k-1}|^p\dx
		\end{align*}
		yield
		\begin{align}
		\label{eqn:eq2}
			\tau\int_\Omega|D_{\tau,k}(\chi)|^2\dx+\int_\Omega\frac1p|\nabla\chi_\tau^k|^p\dx-\int_\Omega\frac1p|\nabla\chi_\tau^{k-1}|^p\dx+R_2\leq 0
		\end{align}
		with the remainder term
		\begin{align*}
			R_2:={}&\int_\Omega \gamma(\chi_\tau^k)(\chi_\tau^k-\chi_\tau^{k-1})\dx
				+\int_\Omega\frac{b_1'(\chi_\tau^{k})
				+b_2'(\chi_\tau^{k-1})}{2}\CC\e(u_\tau^{k-1}):\e(u_\tau^{k-1})(\chi_\tau^k-\chi_\tau^{k-1})\dx\\
			&-\int_\Omega\Theta_M(w_\tau^{k-1})(\chi_\tau^k-\chi_\tau^{k-1})\dx
				-\int_\Omega\rho'(\chi_\tau^{k-1})\Theta_M(w_\tau^{k})\di(u_\tau^{k-1})(\chi_\tau^k-\chi_\tau^{k-1})\dx.
		\end{align*}
		Testing \eqref{eqn:discrHeatEq} with $\tau$ shows
		\begin{align}
		\label{eqn:eq3}
			\int_\Omega\bl w_\tau^k-w_\tau^{k-1}\br\dx+R_3\leq 0
		\end{align}
		with the remainder term
		\begin{align*}
			R_3:={}&\int_\Omega\Theta_M(w_\tau^{k-1})(\chi_\tau^k-\chi_\tau^{k-1})\dx+\int_\Omega\rho(\chi_\tau^{k-1})\Theta_M(w_\tau^k)\di\bl u_\tau^k-u_\tau^{k-1}\br\dx\\
			&+\int_\Omega\rho'(\chi_\tau^{k-1})\Theta_M(w_\tau^k)\di\bl u_\tau^{k-1}\br(\chi_\tau^k-\chi_\tau^{k-1})\dx.
		\end{align*}
		By adding \eqref{eqn:eq1}-\eqref{eqn:eq3}, noticing the crucial property
		$$
			R_1+R_2+R_3=\int_\Omega \gamma(\chi_\tau^k)(\chi_\tau^k-\chi_\tau^{k-1})\dx
		$$
		and summing over $k=1,\ldots,\OL t_\tau/\tau$, we obtain
		\begin{align}
			&\int_\Omega \OL{w}_\tau(t)\dx+\frac 12\|\partial_t u_\tau(t)\|_{L^2(\Omega)}^2
				+c\|\e(\OL u_\tau(t))\|_{L^2(\Omega;\R^{d\times d})}^2+\frac1p\|\nabla\OL\chi_\tau(t)\|_{L^p(\Omega)}^p\notag\\
			&+\int_0^{\OL t_\tau}\Big(c\|\e(\partial_t u_\tau(s))\|_{L^2(\Omega;\R^{d\times d})}^2+\|\partial_t\chi_\tau(s)\|_{L^2(\Omega)}^2\Big)\ds\notag\\
 			&\qquad\leq\int_\Omega w^0\dx+\frac 12\|v^0\|_{L^2(\Omega)}^2+\int_\Omega\frac{b(\chi^0)}{2}\CC\e(u^0):\e(u^0)\dx
				+\frac1p\|\nabla\chi^0\|_{L^p(\Omega)}^p\notag\\
		\label{eqn:aPrioriEst1}
			&\qquad\quad+\int_0^{\OL t_\tau}\int_\Omega -\gamma(\OL\chi_\tau)\partial_t\chi_\tau\dxs.
		\end{align}
		The last term on the right-hand side can be estimated from above as follows:
		\begin{align*}
			\int_0^{\OL t_\tau}\int_\Omega -\gamma(\OL\chi_\tau)\partial_t\chi_\tau\dxs
 			\leq C\int_0^{\OL t_\tau}\|\partial_t\chi(s)\|_{L^2(\Omega)}\ds.
		\end{align*}
		and, therefore, absorbed by the left-hand side.
		Hence the left-hand side of \eqref{eqn:aPrioriEst1}
		is bounded with respect to $\tau$ and $t$.
		\ep
	\end{proof}
%	\begin{lemma}[Second a priori estimate]
%	\label{lemma:second1}
%		It holds
%		$$
%			\sum_{k=0}^l\tau\int_\Omega \nabla w_\tau^{k}\cdot\nabla\bl\Theta_M(w_\tau^k)\br<C.
%		$$
%		for a constant $C>0$ independent of $\tau$ and $l=0,\ldots,T_\tau$.
%	\end{lemma}
	\\
	\begin{proof}[Proof of the second a~priori estimate]
		By testing \eqref{eqn:discrHeatEq} with $\tau \Theta_M(w_\tau^k)$ and using the convexity estimate
		$$
			\Theta_M(w_\tau^k)(w_\tau^k-w_\tau^{k-1})\geq \widehat\Theta_M(w_\tau^k)-\widehat\Theta_M(w_\tau^{k-1}),
		$$
		where $\widehat\Theta_M$ denotes the antiderivative of $\Theta_M$ with $\widehat\Theta_M(0)=0$ (note that $\widehat\Theta_M$ is convex
		due to $\Theta_M'\geq 0$),
		we obtain
		\begin{align*}
			&\int_\Omega\widehat\Theta_M(w_\tau^k)\dx-\int_\Omega\widehat\Theta_M(w_\tau^{k-1})\dx
				+\tau\int_\Omega K_M(w_\tau^{k-1})\nabla w_\tau^{k}\cdot\nabla\bl\Theta_M(w_\tau^k)\br\dx\\
			&\quad\leq-\tau\int_\Omega\Theta_M(w_\tau^{k-1})D_{\tau,k}(\chi)\Theta_M(w_\tau^k)\dx
				-\tau\int_\Omega\rho(\chi_\tau^{k-1})\Theta_M(w_\tau^{k})\di\bl D_{\tau,k}(u)\br\Theta_M(w_\tau^k)\dx\notag\\
			&\quad\quad-\tau\int_\Omega\rho'(\chi_\tau^{k-1})\Theta_M(w_\tau^{k})\di(u_\tau^{k-1})D_{\tau,k}(\chi)\Theta_M(w_\tau^k)\dx\\
			&\quad\leq\tau\|\Theta_M(w_\tau^{k-1})\|_{L^\infty(\Omega)}\|D_{\tau,k}(\chi)\|_{L^2(\Omega)}\|\Theta_M(w_\tau^k)\|_{L^{\infty}(\Omega)}\\
			&\quad\quad+\tau\|\rho(\chi_\tau^{k-1})\|_{L^\infty(\Omega)}\|\Theta_M(w_\tau^{k})\|_{L^\infty(\Omega)}
				\|\di\bl D_{\tau,k}(u)\br\|_{L^2(\Omega)}\|\Theta_M(w_\tau^k)\|_{L^\infty(\Omega)}\\
			&\quad\quad+\tau\|\rho'(\chi_\tau^{k-1})\|_{L^\infty(\Omega)}\|\Theta_M(w_\tau^{k})\|_{L^\infty(\Omega)}\|\di(u_\tau^{k-1})\|_{L^2(\Omega)}\|D_{\tau,k}(\chi)\|_{L^2(\Omega)}
				\|\Theta_M(w_\tau^k)\|_{L^\infty(\Omega)}.
		\end{align*}
		By summing over the time index $k=1,\ldots,T_\tau$, we end up with the estimate
		\begin{align*}
			&\int_\Omega\widehat\Theta_M(\OL w_\tau(T))\dx
				+\int_{\Omega_T} K_M(\UL w_\tau)\nabla \OL w_\tau\cdot\nabla\bl\Theta_M(\OL w_\tau)\br\dxt\\
			&\quad\leq
			\int_\Omega\widehat\Theta_M(w^0)\dx
			+\|\Theta_M(\UL w_\tau)\|_{L^\infty(L^\infty)}\|\partial_t\chi_\tau\|_{L^2(L^2)}\|\Theta_M(\OL w_\tau)\|_{L^2(L^{\infty})}\\
			&\quad\quad+\|\rho(\UL\chi_\tau)\|_{L^\infty(L^\infty)}\|\Theta_M(\OL w_\tau)\|_{L^\infty(L^\infty)}
				\|\di\bl \partial_t u_\tau\br\|_{L^2(L^2)}\|\Theta_M(\OL w_\tau)\|_{L^2(L^\infty)}\\
			&\quad\quad+\|\rho'(\UL\chi_\tau)\|_{L^\infty(L^\infty)}\|\Theta_M(\OL w_\tau)\|_{L^\infty(L^\infty)}\|\di(\UL u_\tau)\|_{L^2(L^2)}\|\partial_t\chi_\tau\|_{L^2(L^2)}
				\|\Theta_M(\OL w_\tau)\|_{L^\infty(L^\infty)}.
		\end{align*}
		Together with the first a~priori estimate,
		the estimate $\Theta_M(\OL w_\tau)\leq \Theta(M)$ (holding uniformly in $\tau$)
		and the Assumptions (A2) and (A3), we obtain boundedness of
		\begin{align*}
			c_1\int_{\Omega_T}|\nabla \OL w_\tau|^2\Theta_M'(\OL w_\tau)\dxt
			&\leq\int_{\Omega_T} K_M(\UL w_\tau)|\nabla \OL w_\tau|^2\Theta_M'(\OL w_\tau)\dxt\\
			&= \int_{\Omega_T} K_M(\UL w_\tau)\nabla \OL w_\tau\cdot\nabla\bl\Theta_M(\OL w_\tau)\br\dxt.
		\end{align*}
		Since $\Theta_M'(w_\tau)$ is also bounded in $L^\infty(0,T;L^\infty(\Omega))$ by the Lipschitz continuity of $\Theta$ (see Assumption (A2)), we obtain the claim as follows:
		\begin{align*}
			\|\nabla\Theta_M(\overline w_\tau)\|_{L^2(L^2)}^2
			&= \int_{\Omega_T} |\nabla \OL w_\tau|^2(\Theta_M'(\OL w_\tau))^2\dxt\\
			&\leq \|\Theta_M'(w_\tau)\|_{L^\infty(L^\infty)}\int_{\Omega_T} |\nabla \OL w_\tau|^2\Theta_M'(\OL w_\tau)\dxt.
		\end{align*}
%		see that the r.h.s. is bounded due to the first a priori estimate.
		\ep
	\end{proof}
%	\begin{lemma}[Third a priori estimate]
%	\label{lemma:thirdAPE}
%		It holds
%		\begin{align*}
%			&\sum_{k=0}^l\tau\|(D_{\tau,k}(u)\|_{H^2(\Omega;\R^d)}\leq C,
%			&&\|D_{\tau,k}(u)\|_{H^1(\Omega;\R^d)}\leq C,\\
%			&\sum_{k=0}^l\tau\|D_{\tau,k}(D_{\tau,k}(u))\|_{L^2(\Omega;\R^d)}\leq C
%		\end{align*}
%		for a constant $C>0$ independent of $\tau$.
%	\end{lemma}
	\\
	\begin{proof}[Proof of the third a~priori estimate]
		We test equation \eqref{eqn:momentumEq} with $-\tau\di(\DD\e(D_{\tau,k}(u)))$ and
		sum over $k=1,\ldots,\OL t_\tau$ for a chosen $t\in[0,T]$.
		The corresponding calculations without the term
		$$
			\int_0^{\OL t_\tau}\int_\Omega\di\bl\rho(\OL\chi_\tau)\Theta_M(\OL w_\tau)\mathbf 1\br\cdot\di(\DD\e(\partial_t u_\tau))\dxs
		$$
		are carried out in  \cite[Proposition 3.10]{RR12} (see also \cite[Fifth Estimate, p. 18]{RR14}).
		The calculations there take advantage of Assumption (A8), i.e. $p>d$.
		
		In our case, we have to estimate the additional term
		where the $\chi$-dependence of $\rho$ comes into play
		\begin{align*}
			&\left|\int_0^{\OL t_\tau}\int_{\Omega}\di\big(\rho(\OL\chi_\tau)\Theta_M(\OL w_\tau)\mathbf 1\big)\cdot\di(\DD\e(\partial_t u_\tau))\dxs\right|\\
			&\qquad\leq
				\int_0^{\OL t_\tau}\int_{\Omega}\left|\Theta_M(\OL w_\tau)\rho'(\OL\chi_\tau)\nabla\OL\chi_\tau\cdot\di\bl\DD\e\bl\partial_t u_\tau\br\br\right|\dxs\\
			&\qquad\quad+\int_0^{\OL t_\tau}\int_{\Omega}\left|\rho(\OL\chi_\tau)\nabla\big(\Theta_M(\OL w_\tau)\big)\cdot\di\bl\DD\e\bl \partial_t u_\tau\br\br\right|\dxs\\
			&\qquad\leq C\|\Theta_M(\OL w_\tau)\|_{L^\infty(L^\infty)}\|\rho'(\OL\chi_\tau)\|_{L^\infty(L^\infty)}
				\|\nabla\OL\chi_\tau\|_{L^2(L^2)}\left\|\partial_t u_\tau\right\|_{L^2(H^2)}\\
			&\qquad\quad+C\|\rho(\OL\chi_\tau)\|_{L^\infty(L^\infty)}\|\nabla\big(\Theta_M(\OL w_\tau)\big)\|_{L^2(L^2)}
				\left\|\partial_t u_\tau\right\|_{L^2(H^2)}.
		\end{align*}
		
		By using the first and the second a~priori estimates,
		the estimate $\Theta_M(\OL w_\tau)\leq \Theta(M)$ (holding uniformly in $\tau$)
		and the regularity estimate for linear elasticity (cf. \cite[Lemma 3.2]{Necas67})
		\begin{align*}
			\|u\|_{H^2}\leq C\|\di(\DD\e(u))\|_{L^2}\qquad\text{for all } u\in H_0^2(\Omega;\R^d)
		\end{align*}
		as well as the calculations in \cite[Proposition 3.10]{RR12} (see also \cite[Fifth Estimate, p. 18]{RR14}), we obtain eventually for small $\delta>0$:
		\begin{align*}
			&\frac12\|\e(\partial_t u_\tau(t))\|_{L^2(\Omega;\R^{d\times d})}^2+\delta\|\partial_t u_\tau\|_{L^2(0,\OL t_\tau;H^2(\Omega;\R^d))}^2\\
			&\qquad\leq \frac12\|\e(v^0)\|_{L^2(\Omega;\R^{d\times d})}^2+C\int_0^{\OL t_\tau}\|\partial_t u_\tau\|_{L^2(0,\OL s_\tau;H^2(\Omega;\R^d))}^2\ds
				+C\left\|\partial_t u_\tau\right\|_{L^2(0,\OL t_\tau;H^2(\Omega;\R^d))}.
		\end{align*}
		Gronwall's lemma leads to the claim.
		\ep
	\end{proof}
%	\begin{lemma}[Fourth a priori estimate]
%		It holds
%		\begin{align*}
%			&\|w_\tau^l\|_{L^2(\Omega)}\leq C,
%			&&\sum_{k=0}^l\tau\|\nabla w_\tau^k\|_{L^2(\Omega)}\leq C
%		\end{align*}
%		for a constant $C>0$ independent of $\tau$ and $l=0,\ldots,T_\tau$.
%	\end{lemma}
	\\
	\begin{proof}[Proof of the fourth a~priori estimate]
		Testing \eqref{eqn:discrHeatEq} with $\tau w_\tau^k$ and using standard convexity estimates as well as Assumption (A3) yield
		\begin{align*}
			&\frac 12\|w_\tau^k\|_{L^2(\Omega)}^2-\frac 12\|w_\tau^{k-1}\|_{L^2(\Omega)}^2+c_1\tau\|\nabla w_\tau^k\|_{L^2(\Omega;\R^d)}^2\\
			&\qquad\leq
				\tau\|\Theta_M(w_\tau^{k-1})w_\tau^k D_{\tau,k}(\chi)\|_{L^1(\Omega)}
				+C\tau\|\rho(\chi_\tau^{k-1})\Theta_M(w_\tau^k)\|_{L^\infty(\Omega)}\|\di\big(D_{\tau,k}(u)\big)w_\tau^k\|_{L^1(\Omega)}\\
			&\qquad\quad+C\tau\|\rho'(\chi_\tau^{k-1})\Theta_M(w_\tau^k)\|_{L^\infty(\Omega)}\|D_{\tau,k}(\chi)\di (u_\tau^{k-1})w_\tau^k\|_{L^1(\Omega)}\leq 0.
		\end{align*}
%		\textcolor{red}{To estimate the last term, we need the growth condition $\Theta_M(x)x\leq C$ for all $x\geq 0$.}
		
		Summing over the discrete time index $k=1,\ldots,\OL t_\tau/\tau$, using the continuous embedding $H^1(\Omega)\hookrightarrow L^6(\Omega)$
		and standard estimates, we receive
		\begin{align*}
			&\frac 12\|\OL w_\tau(t)\|_{L^2(\Omega)}^2+c_1\|\nabla \OL w_\tau\|_{L^2(0,\OL t_\tau,L^2(\Omega;\R^d))}^2\\
			&\quad\leq 
				\frac 12\|w^0\|_{L^2(\Omega)}^2+C\|\Theta_M(\UL w_\tau)\|_{L^\infty(L^\infty)}\left(\int_0^{\OL t_\tau}\|\OL w_\tau(s)\|_{L^2(\Omega)}^2\ds
				+\|\partial_t\chi_\tau\|_{L^2(L^2)}^2\right)\\
			&\quad\quad
				+C\|\rho(\UL\chi_\tau)\|_{L^\infty(L^\infty)}\|\Theta_M(\OL w_\tau)\|_{L^\infty(L^\infty)}
				\left(\int_0^{\OL t_\tau}\|\OL w_\tau(s)\|_{L^2(\Omega)}^2\ds
				+\|\di\big(\partial_t u_\tau\big)\|_{L^2(L^2)}^2\right)\\
			&\quad\quad
				+\|\rho'(\UL\chi_\tau)\|_{L^\infty(L^\infty)}\|\Theta_M(\OL w_\tau)\|_{L^\infty(L^\infty)}\times\\
			&\qquad\qquad\times
				\left(\delta\int_0^{\OL t_\tau}\|\OL w_\tau(s)\|_{H^1(\Omega)}^2\ds
				+C_\delta\|\partial_t \chi_\tau\|_{L^2(L^2)}^2\|\di(\UL u_\tau)\|_{L^\infty(L^3)}^2\right).
		\end{align*}
		Chosing $\delta>0$ sufficiently small, applying the first and the third a~priori estimates
		as well as the estimate $\Theta_M(\OL w_\tau)\leq \Theta(M)$,
		we obtain by Gronwall's inequality boundedness of the left-hand side and, therefore, the claim.
		\ep
	\end{proof}\\
	\begin{proof}[Proof of the fifth a~priori estimate]
		A comparison argument in equation \eqref{eqn:discrHeatEq}
		shows the assertion.\ep
	\end{proof}
	
	\subsection{The passage $\tau\downarrow 0$}
	\label{section:tauToZero}
	By utilizing Lemma \ref{lemma:aPrioriDiscr} and by noticing $\partial_t u_\tau=\overline{v}_\tau$ (see \eqref{eqn:vDef}), we obtain
	by standard compactness and Aubin-Lions type theorems (see \cite{Simon}) the following convergence properties.
	\begin{corollary}
	\label{cor:weakConvDiscr}
		We obtain functions $(u,w,\chi)$ which are in the spaces
		\begin{align*}
			&u\in H^1(0,T;H_0^2(\Omega;\R^d))\cap W^{1,\infty}(0,T;H_0^1(\Omega;\R^d))\cap H^{2}(0,T;L^2(\Omega;\R^d))\\
			&\quad \text{ with }u(0)=u^0\text{ a.e. in }\Omega,\;\partial_t u(0)=v^0\text{ a.e. in }\Omega,\\
			&w\in L^2(0,T;H^1(\Omega))\cap L^\infty(0,T;L^2(\Omega))\cap H^1(0,T;H^1(\Omega)^*)\\
			&\quad \text{ with }w(0)=w^0\text{ a.e. in }\Omega,\;w\geq 0\text{ a.e. in }\Omega_T,\\
			&\chi\in L^\infty(0,T;W^{1,p}(\Omega))\cap H^1(0,T;L^2(\Omega))\\
			&\quad \text{ with }\chi(0)=\chi^0\text{ a.e. in }\Omega,\;\chi\geq 0\text{ a.e. in }\Omega_T,\;\partial_t\chi\leq 0\text{ a.e. in }\Omega_T
%			&\xi\in L^1(0,T;L^1(\Omega))
		\end{align*}
		such that (along a subsequence) for all $\varepsilon\in(0,1]$, $\mu\geq 1$:
		\begin{align*}
			\textit{(i) }&u_\tau\to u&&\textit{ weakly-star in }H^1(0,T;H^2(\Omega;\R^d))\cap W^{1,\infty}(0,T;H^1(\Omega;\R^d)),\\
			&\overline{u}_\tau,\underline{u}_\tau\to u&&\textit{ weakly-star in }L^\infty(0,T;H^2(\Omega;\R^d)),\\
			&u_\tau\to u&&\textit{ strongly in }H^1(0,T;H^{2-\varepsilon}(\Omega;\R^d)),\\
			&\overline{u}_\tau,\underline{u}_\tau\to u&&\textit{ strongly in }L^\infty(0,T;H^{2-\varepsilon}(\Omega;\R^d)),\\
			&u_\tau, \overline{u}_\tau,\underline{u}_\tau\to u&&\textit{ a.e. in }\Omega_T,\\
			\textit{(ii) }&v_\tau\to \partial_t u&&\textit{ weakly-star in }H^1(0,T;L^2(\Omega;\R^d)),\\
			&v_\tau\to \partial_t u&&\textit{ strongly in }L^2(0,T;H^{2-\varepsilon}(\Omega;\R^d)),\\
			\textit{(iii) }&w_\tau\to w&&\textit{ weakly-star in }L^2(0,T;H^1(\Omega))\cap L^\infty(0,T;L^2(\Omega))\\
			&&&\hspace*{6.6em}\cap H^1(0,T;H^1(\Omega)^*),\\
				&\overline{w}_\tau,\underline{w}_\tau\to w&&\textit{ weakly-star in }L^2(0,T;H^1(\Omega))\cap L^\infty(0,T;L^2(\Omega)),\\
				&w_\tau,\overline{w}_\tau,\underline{w}_\tau\to w&&\textit{ strongly in }L^2(0,T;H^{1-\varepsilon}(\Omega))\cap L^\mu(0,T;L^2(\Omega)),\\
				&w_\tau,\overline{w}_\tau,\underline{w}_\tau\to w&&\textit{ a.e. in }\Omega_T,\\
			\textit{(iv) }&\chi_\tau\to\chi&&\textit{ weakly-star in }L^\infty(0,T;W^{1,p}(\Omega))\cap H^1(0,T;L^2(\Omega)),\\
				&\overline{\chi}_\tau,\underline{\chi}_\tau\to\chi&&\textit{ weakly-star in }L^\infty(0,T;W^{1,p}(\Omega)),\\
				&\chi_\tau,\overline{\chi}_\tau,\underline{\chi}_\tau\to\chi&&\textit{ strongly in }L^\mu(0,T;W^{1-\varepsilon,p}(\Omega)),\\
				&\overline{\chi}_\tau,\underline{\chi}_\tau\to\chi&&\textit{ uniformly on }\OL{\Omega_T}.
		\end{align*}
	\end{corollary}
	\textbf{Comments on the proof of Corollary \ref{cor:weakConvDiscr}.}\\
		We would like to make some comments about the strong convergences:
		
		Recalling the boundedness of $v_\tau$ in $H^1(0,T;L^2(\Omega;\R^d))\cap L^2(0,T;H^2(\Omega;\R^d))$ from Lemma \ref{lemma:aPrioriDiscr}
		and the compact embedding $H^{2}(\Omega;\R^d)\hookrightarrow H^{2-\varepsilon}(\Omega;\R^d)$
		(where $H^{2-\varepsilon}$ denotes the Sobolev-Slobodeckij space of fractional order $2-\varepsilon$, cf. \cite[p. 18]{Rou13}), we obtain from
		Aubin-Lions' theorem
		\begin{align}
		\label{eqn:vConvergence}
			&v_\tau\to v&&\text{ strongly in }L^2(0,T;H^{2-\varepsilon}(\Omega;\R^d))
		\end{align}
		as $\tau\downarrow 0$ for a subsequence.
		Similiarly we obtain
		\begin{subequations}
		\begin{align}
		\label{eqn:uConvergence}
			&u_\tau\to u&&\text{ strongly in }L^2(0,T;H^{2-\varepsilon}(\Omega;\R^d)),\\
		\label{eqn:wConvergence}
			&w_\tau\to w&&\text{ strongly in }L^2(0,T;H^{1-\varepsilon}(\Omega)),\\
		\label{eqn:chiConvergence}
			&\chi_\tau\to \chi&&\text{ strongly in }L^\mu(0,T;W^{1-\varepsilon,p}(\Omega)).
		\end{align}
		\end{subequations}
		These features also imply the corresponding convergence properties for $\OL u_\tau$, $\UL u_\tau$, $\OL v_\tau$, $\UL v_\tau$,
		$\OL w_\tau$, $\UL w_\tau$, $\OL \chi_\tau$ and $\UL \chi_\tau$.
		Noticing \eqref{eqn:vConvergence}, \eqref{eqn:uConvergence} and $\partial_t u_\tau=\OL v_\tau$ show
		\begin{align*}
			&u_\tau\to u&&\text{ strongly in }H^1(0,T;H^{2-\varepsilon}(\Omega;\R^d)).
		\end{align*}
		Resorting to a further suitable subsequence \eqref{eqn:wConvergence} implies
		$\|w_\tau(t)-w(t)\|_{L^2}\to 0$ as $\tau\downarrow 0$ for a.e. $t\in(0,T)$.
		By using $\|w_\tau(t)-w(t)\|_{L^2}^\mu\leq \|w_\tau-w\|_{L^\infty(L^2)}^\mu\leq C$, Lebesgue's convergence theorem yields
		\begin{align*}
			&w_\tau\to w&&\text{ strongly in }L^\mu(0,T;L^2(\Omega)).\qquad\quad\;
		\end{align*}
		A further application of a Aubin-Lions type theorem is
		\begin{align*}
			&\chi_\tau\to\chi&&\text{ uniformly on }\OL{\Omega_T}\qquad\qquad\qquad\quad\;\;\;
		\end{align*}
		as $\tau\downarrow 0$ for a subsequence
		which follows from boundedness of $\chi_\tau$ in $L^\infty(0,T;W^{1,p}(\Omega))\cap H^1(0,T;L^2(\Omega))$
		and the compact embedding $W^{1,p}(\Omega)\hookrightarrow C(\OL{\Omega_T})$ valid for $p>d$.
		\ep
	\begin{lemma}
	\label{lemma:strongZConv}
		It even holds (along a subsequence as $\tau\downarrow 0$)
		$$
			\overline{\chi}_\tau\to\chi\text{ strongly in }L^p(0,T;W^{1,p}(\Omega)).
		$$
	\end{lemma}
	\begin{proof}
		Applying an approximations result from \cite[Lemma 5.2]{WIAS1520}, we obtain a sequence $\{\zeta_\tau\}$ in the space $L^p(0,T;W_+^{1,p}(\Omega))$
%		and constants $\nu_\tau(t)>0$
		such that $\zeta_\tau\to \chi$ in $L^p(0,T;W^{1,p}(\Omega))$ as $\tau\downarrow 0$ and
		$$
			0\leq \zeta_\tau(t)\leq\underline{\chi}_\tau(t)\text{ a.e. in }\Omega_T.
		$$
		
		The claim can now be shown by using a uniform monotonicity estimate of the $L^p$-norm
		\begin{align*}
			\|\nabla\overline{\chi}_\tau-\nabla\chi\|_{L^p(\Omega_T)}^p
				\leq{}& C\int_{\Omega_T}\bl|\nabla\overline{\chi}_\tau|^{p-2}\nabla\overline{\chi}_\tau-|\nabla\chi|^{p-2}\nabla\chi\br\cdot\nabla(\overline{\chi}_\tau-\chi)\dxt\\
				={}&C\int_{\Omega_T}\bl|\nabla\overline{\chi}_\tau|^{p-2}\nabla\overline{\chi}_\tau-|\nabla\chi|^{p-2}\nabla\chi\br\cdot\nabla(\overline{\chi}_\tau-\zeta_\tau)\dxt\\
				&+C\int_{\Omega_T}\bl|\nabla\overline{\chi}_\tau|^{p-2}\nabla\overline{\chi}_\tau-|\nabla\chi|^{p-2}\nabla\chi\br\cdot\nabla(\zeta_\tau-\chi)\dxt,
		\end{align*}
		by applying Corollary \ref{cor:weakConvDiscr} and by testing the variational inequality \eqref{eqn:damageVIdiscr} with $\zeta_\tau$,
		it can be shown that $lim\,sup$ of the right-hand side is $\leq 0$.
		\ep
	\end{proof}
	
	The passage to the limit $\tau\downarrow 0$ in the time-discrete system in Lemma \ref{lemma:discrSys} can now be performed as follows.
	\begin{itemize}
			\item
				\textbf{Heat equation.}
				Integrating equation \eqref{eqn:discrHeatEq} over the time $[0,T]$, Corollary \ref{cor:weakConvDiscr} allows to pass to the limit $\tau\downarrow 0$ by
				taking into account the uniform boundedness of $K_M(\UL w_\tau)$, $\Theta_M(\UL w_\tau)$ and $\Theta_M(\OL w_\tau)$ in $L^\infty(\Omega)$.
				Then, by switching to an a.e. $t$ formulation in the limit, we obtain for every $\zeta\in H^{1}(\Omega)$ and a.e. $t\in(0,T)$:
				\begin{align}
					&\langle\partial_t w,\zeta\rangle_{H^1}+\int_\Omega\bl K_M(w)\nabla w\cdot\nabla\zeta
						+\Theta_M(w)\partial_t \chi\zeta\br\dx\notag\\
					&\quad+\int_\Omega\Big(\rho(\chi)\Theta_M(w)\di\bl \partial_t u\br\zeta
					+\rho'(\chi)\Theta_M(w)\di(u)\partial_t\chi\zeta\Big)\dx=0.
				\label{eqn:heatEq2}
				\end{align}
			\item
				\textbf{Balance of forces.}
				To obtain the equation for the balance of forces, we integrate equation \eqref{eqn:discrMomentumEq} over $\Omega_T$ and use
				Corollary \ref{cor:weakConvDiscr} to pass to the limit $\tau\downarrow 0$.
				In the limit we have the necessary regularity properties to switch to an a.e. formulation in $\Omega_T$, i.e.
				it holds
				\begin{align}
				\label{eqn:momentumEq}
					&\partial_{tt}u-\di\bl b(\chi)\e(u)\br-\di\bl a(\chi)\e(\partial_t u)\br
					+\di\bl\rho(\chi)\Theta_M(w)\mathbf 1\br=0
				\end{align}
				a.e. in $\Omega_T$.
			\item
				\textbf{One-sided variational inequality for the damage process.}
				The limit passage for equation \eqref{eqn:damageVIdiscr} can be accomplished by an approximation argument developed in \cite{WIAS1520}.
				Note that this approach strongly relies on $p>d$ (see (A8)).
				We sketch the argument.
				\begin{itemize}
					\item
						Initially, the main idea is to consider time-depending test-functions $\Psi\in L^\infty(0,T;W_-^{1,p}(\Omega))$
						which satisfy for a.e. $t\in(0,T)$ the constraint
						$$
							\{x\in\OL\Omega\,|\,\Psi(x,t)=0\}\supseteq\{x\in\OL\Omega\,|\,\chi(x,t)=0\}.
						$$
						Here, we make use of the embedding $W^{1,p}(\Omega)\hookrightarrow C(\OL\Omega)$.
					\item
						As shown in \cite[Lemma 5.2]{WIAS1520}, we obtain an approximation sequence $\{\Psi_\tau\}\subseteq L^p(0,T;W_-^{1,p}(\Omega))$
						and constants $\nu=\nu(\tau,t)>0$ (independent of $x$) such that
						$\Psi_\tau\to\Psi$ in $L^p(0,T;W^{1,p}(\Omega))$ as $\tau\downarrow 0$ and
						$0\leq -\nu\Psi_\tau(t)\leq\overline{\chi}_\tau(t)$ in $\Omega$ for a.e. $t\in(0,T)$.
						Multiplying this inequality by -1, adding $\overline{\chi}_\tau(t)$
						and using the monotonicity condition $\overline{\chi}_\tau\leq \underline{\chi}_\tau$,
						we obtain
						\begin{align}
						\label{eqn:tfEst}
							0\leq \nu\Psi_\tau(t)+\overline{\chi}_\tau(t)\leq \underline{\chi}_\tau(t)\text{ in }\Omega.
						\end{align}
					\item
						Because of \eqref{eqn:tfEst}, we are allowed to test \eqref{eqn:damageVIdiscr} with $\nu_\tau(t)\Psi_\tau(t)+\overline{\chi}_\tau(t)$.
						Dividing the resulting inequality by $\nu$ (which is positive and independent of $x$), integrating
						in time over $[0,T]$, passing to the limit and switching back to an a.e. $t$ formulation, we obtain
						for a.e. $t\in(0,T)$
						\begin{align*}
							&0\leq\int_\Omega\bl \partial_t \chi\zeta+|\nabla\chi|^{p-2}\nabla\chi\cdot\nabla\zeta+\gamma(\chi)\zeta
								+\frac{b'(\chi)}{2}\CC\e(u):\e(u)\zeta\br\dx\notag\\
							&\qquad+\int_\Omega\bl-\Theta_M(w)\zeta
								-\rho'(\chi)\Theta_M(w)\di(u)\zeta\br\dx,
						\end{align*}
						for all $\zeta\in W_-^{1,p}(\Omega)$ with $\{\zeta=0\}\supseteq\{\chi(t)=0\}$.
					\item
						It is shown in \cite[Lemma 5.3]{WIAS1520} that, in this case, we obtain
						\begin{align}
							&0\leq\int_\Omega\bl \partial_t \chi\zeta+|\nabla\chi|^{p-2}\nabla\chi\cdot\nabla\zeta
								+\gamma(\chi)\zeta+\frac{b'(\chi)}{2}\CC\e(u):\e(u)\zeta\br\dx\notag\\
						\label{eqn:damageEq}
							&\qquad+\int_\Omega\bl-\Theta_M(w)\zeta
								-\rho'(\chi)\Theta_M(w)\di(u)\zeta+\xi\zeta\br\dx
						\end{align}
						for all $\zeta\in W_-^{1,p}(\Omega)$ and for a.e. $t\in(0,T)$, where $\xi\in L^2(0,T;L^2(\Omega))$ is given by
						\begin{align}
							&\xi:=-\mathbf 1_{\{\chi=0\}}\Big(\gamma(\chi)+\frac{b'(\chi)}{2}\CC\e(u):\e(u)-\Theta_M(w)
							-\rho'(\chi)\Theta_M(w)\di(u)\Big)^+,
						\label{eqn:defXiM}
						\end{align}
						with $(\cdot)^+:=\max\{\cdot,0\}$.
						Note that $\partial_t \chi$ does not appear in the bracket.
						In particular, $\xi$ fulfills
						\begin{align}
							&\int_\Omega\xi(\zeta-z)\dx\leq 0
						\label{eqn:damageEq2}
						\end{align}
						for all $\zeta\in W_+^{1,p}(\Omega)$ and a.e. $t\in(0,T)$.
				\end{itemize}
			\item
				\textbf{Partial energy inequality.}
				Testing the variational inequality \eqref{eqn:damageVIdiscr} with $\chi_\tau^k-\chi_\tau^{k-1}$ and applying the convexity argument
				$$
				\int_\Omega|\nabla\chi_\tau^k|^{p-2}\nabla\chi_\tau^k\cdot\nabla(\chi_\tau^k-\chi_\tau^{k-1})\dx
					\geq \int_\Omega\frac 1p|\nabla \chi_\tau^k|^p\dx-\int_\Omega\frac 1p|\nabla \chi_\tau^{k-1}|^p\dx
				$$
				and summing over the time index $k=\UL s_\tau/\tau+1,\ldots,\OL t_\tau/\tau$, we end up with
				\begin{align*}
					&\int_\Omega\frac 1p|\nabla \overline{\chi}_\tau(t)|^p\dx-\int_\Omega\frac 1p|\nabla \OL\chi_\tau(s)|^p\dx\\
					&\qquad+\int_{\UL s_\tau}^{\OL t_\tau}\int_\Omega\Big(\gamma(\overline{\chi}_\tau)
					+\frac{b'(\overline{\chi}_\tau)}{2}\CC\e(\underline{u}_\tau):\e(\underline{u}_\tau)\Big)\partial_t\chi_\tau\dx\diota\\
					&\qquad+\int_{\UL s_\tau}^{\OL t_\tau}\int_\Omega\Big(-\Theta_M(\underline{w}_\tau)
					-\rho'(\underline{\chi}_\tau)\Theta_M(\overline{w}_\tau)\di(\underline{u}_\tau)
					+\partial_t\chi_\tau\Big)\partial_t\chi_\tau\dx\diota\leq 0
%				\label{eqn:EI}
				\end{align*}
				for a.e. $t\in(0,T)$.
				Passing to the limit $\tau\downarrow 0$ by using Corollary \ref{cor:weakConvDiscr},
				weakly lower-semicontinuity arguments and the estimate $t\leq \OL t_\tau$ and $s\geq \UL s_\tau$ for the quadratic term in $\partial_t\chi$,
				we get for a.e. $0\leq s\leq t\leq T$ and for a.e. $t\in(0,T)$ with $s=0$ the desired partial energy inequality
				\begin{align}
					&\int_\Omega\frac 1p|\nabla \chi(t)|^p\dx-\int_\Omega\frac 1p|\nabla \chi(s)|^p\dx
						+\int_{s}^{t}\int_\Omega\Big(\gamma(\chi)+\frac{b'(\chi)}{2}\CC\e(u):\e(u)\Big)\partial_t\chi\dx\diota\notag\\
					&\qquad+\int_{s}^{t}\int_\Omega\Big(-\Theta_M(w)-\rho'(\chi)\Theta_M(w)\di(u)+\partial_t\chi\Big)\partial_t\chi\diota\leq 0.
				\label{eqn:EI}
				\end{align}
		\end{itemize}
		In conclusion, we have proven existence of weak solutions to the truncated system given by \eqref{eqn:heatEq2}, \eqref{eqn:momentumEq},
		\eqref{eqn:damageEq}, \eqref{eqn:damageEq2} and \eqref{eqn:EI}.
		
	\section{Existence of weak solutions for the limit system}
	\label{section:existence2}
	In this section, we will perform the limit analysis for weak solutions of the truncated system as $M\uparrow\infty$.
	We consider for each $M\in\N$ a weak solution $(u_M,w_M,\chi_M,\xi_M)$ as proven in the previous section.
	\subsection{A priori estimates}
	The boundedness properties for $(u_M,w_M,\chi_M,\xi_M)$ uniformly in $M$ are
	based on six different types of a~priori estimates.
		An important ingredient for this estimation series is the assumption
		$1/\sigma\leq 2q-1$ (see (A3)) which is utilized in the subsequent second a priori estimate.
	\begin{lemma}[A priori estimates independent of $M$]
	\label{lemma:aPriori}
	The following boundedness\linebreak properties with respect to $M$ are satisfied:
	\begin{align*}
		\textit{(i)}&\textit{ First a priori estimate:}\notag\\
		&\quad\{u_M\}&&\text{ in }H^1(0,T;H^1(\Omega;\R^d))\cap W^{1,\infty}(0,T;L^2(\Omega;\R^d)),\\
		&\quad\{w_M\}&&\text{ in }L^\infty(0,T;L^1(\Omega)),\\
		&\quad\{\chi_M\}&&\text{ in }L^\infty(0,T;W^{1,p}(\Omega))\cap H^1(0,T;L^2(\Omega)),\\
		\textit{(ii)}&\textit{ Second a priori estimate:}\notag\\
		&\quad\{\C T_M(w_M)\}&&\text{ in }L^2(0,T;H^1(\Omega)),\\
		\textit{(iii)}&\textit{ Third a priori estimate:}\notag\\
		&\quad\{u_M\}&&\text{ in }H^1(0,T;H^2(\Omega;\R^d))\cap W^{1,\infty}(0,T;H^1(\Omega;\R^d))\notag\\
		&&&\quad\;\;\cap H^{2}(0,T;L^2(\Omega;\R^d)),\\
		\textit{(iv)}&\textit{ Fourth a priori estimate:}\notag\\
		&\quad\{\C T_M(w_M)\}&&\text{ in }L^\infty(0,T;L^2(\Omega))\cap L^{2(q+1)}(0,T;L^{6(q+1)}(\Omega)),\\
		\textit{(v)}&\textit{ Fifth a priori estimate:}\notag\\
		&\quad\{w_M\}&&\text{ in }L^\infty(0,T;L^2(\Omega))\cap L^2(0,T;H^1(\Omega)),\\
		\textit{(vi)}&\textit{ Sixth a priori estimate:}\notag\\
		&\quad\{w_M\}&&\text{ in }W^{1,r}(0,T;W_\nu^{2,s}(\Omega)^*)
	\end{align*}
	
	with the constants
	$r:=(2q+2)/(2q_0+1)$ and $s:=(6q+6)/(6q-2q_0+5)$.
	\end{lemma}
%	\begin{lemma}[First a priori estimate]
%	\label{lemma:firstAPE2}
%		It holds $w_M\geq 0$ a.e. in $\Omega_T$ and
%		\begin{align*}
%			&\|w_M\|_{L^\infty(0,T;L^1(\Omega))}\leq C,
%			&&\|\partial_t u_M\|_{L^2(\Omega_T;\R^d)}\leq C,\\
%			&\|u_M\|_{L^\infty(0,T;H^1(\Omega;\R^d))\cap H^1(0,T;H^1(\Omega;\R^d))}\leq C,
%			&&\|\chi_M\|_{L^\infty(0,T;W^{1,p}(\Omega))\cap H^1(0,T;L^2(\Omega))}\leq C.
%		\end{align*}
%		for a constant $C>0$ independent of $\tau$ and $k=0,\ldots,T_\tau$.
%	\end{lemma}
	\begin{proof}[Proof of the first a~priori estimate]
	The first a priori estimate in Lemma \ref{lemma:aPrioriDiscr}
	which is based on the energy estimate \eqref{eqn:aPrioriEst1} is also independent of $M$.
	Lower semi-continuity arguments show the energy estimate also for weak solutions $(u_M,w_M,\chi_M)$ of the time-continuous, truncated system.
	\ep
	\end{proof}
%	\begin{lemma}[Second a priori estimate]
%		It holds
%		\begin{align*}
%			\|\C T_M(w_M)\|_{L^2(0,T;H^1(\Omega))}\leq C.
%		\end{align*}
%	\end{lemma}
	\\
	\begin{proof}[Proof of the second a~priori estimate]
		We deduce the desired estimate by testing \eqref{eqn:heatEq2} with the test-function
		\begin{align}
		\label{eqn:testFct}
			\zeta_M(t)=-(\C T_M(w_M(t))+1)^{-\alpha}\in H^1(\Omega),
		\end{align}
		where $\alpha$ is a fixed real number satisfying $1/\sigma\leq \alpha\leq 2q-1$
		(recap Assumption (A3)).
		We remind that $\C T_M(w_M(t))\geq 0$ a.e. in $\Omega$.
		Integration in time reveals
		\begin{align}
			&\int_0^T\big\langle\partial_t w_M,-(\C T_M(w_M)+1)^{-\alpha}\big\rangle_{H^1}\dt+\int_{\Omega_T}\frac{K_M(w_M)}{(\C T_M(w_M)+1)^{\alpha+1}}\nabla w_M\cdot\nabla \C T_M(w_M)\dxt\notag\\
			&+\int_{\Omega_T}\Big(\partial_t \chi_M+\rho(\chi_M)\di\bl \partial_t u_M\br
				+\rho'(\chi_M)\di(u_M)\partial_t\chi_M\Big)\frac{-\Theta_M(w_M)}{(\C T_M(w_M)+1)^\alpha}\dxt=0.
		\label{eqn:heatEq3}
		\end{align}
		The integral terms on the left-hand side are transformed/estimated in the following calculations.
		\begin{itemize}
			\item
				Let $\psi_M$ denote the function
				$$
					\psi_M(t):=\int_0^t\zeta_M(s)\ds.
				$$
				The use of a generalized chain-rule yields
				\begin{align*}
					\int_0^T\big\langle\partial_t w_M,-(\C T_M(w_M)+1)^{-\alpha}\big\rangle\dt=\int_\Omega\psi_M(w_M(T))\dx-\int_\Omega\psi_M(w^0)\dx.
				\end{align*}
			\item
				By utilizing the identities $\nabla w_M\cdot\nabla \C T_M(w_M)=|\nabla \C T_M(w_M)|^2$
				and $K_M(w_M)=K(\C T_M(w_M))$, the growth assumption for $K$ (see Assumption (A3)),
				the elementary estimate
				$$
					C(a^s+b^s)\geq (a+b)^s\quad\text{for all $a,b\in[0,\infty)$ and constants $C>0$ and $s\geq 1$}
				$$
				(in the sequel we will choose $a=\C T_M(w_M)$, $b=1$ and $s=2q$ which is greater $1$ by (A3))
				as well as $\alpha\leq 2q-1$, we obtain
				\begin{align*}
				\begin{split}
					&\int_{\Omega_T}\frac{K_M(w_M)}{(\C T_M(w_M)+1)^{\alpha+1}}\nabla w_M\cdot\nabla \C T_M(w_M)\dxt\\
					&\qquad=\int_{\Omega_T}\frac{K(\C T_M(w_M))}{(\C T_M(w_M)+1)^{\alpha+1}}|\nabla \C T_M(w_M)|^2\dxt\\
					&\qquad\geq c_1\int_{\Omega_T}\frac{(\C T_M(w_M)^{2q}+1)}{(\C T_M(w_M)+1)^{\alpha+1}}|\nabla \C T_M(w_M)|^2\dxt\\
					&\qquad\geq \widetilde c_1\int_{\Omega_T}\frac{(\C T_M(w_M)+1)^{2q}}{(\C T_M(w_M)+1)^{\alpha+1}}|\nabla \C T_M(w_M)|^2\dxt\\
					&\qquad\geq \widetilde c_1\|\nabla \C T_M(w_M)\|_{L^2(\Omega_T;\R^d)}^2.
					\end{split}
				\end{align*}
			\item
				The identity $\Theta_M(w_M)=\Theta(\C T_M(w_M))$, the growth assumption for $\Theta$ (see Assumption (A2))
				and the estimate $1/\sigma\leq \alpha$ imply boundedness of
				$$
					\left|\frac{\Theta_M(w_M)}{(\C T_M(w_M)+1)^\alpha}\right|
					=\frac{\Theta(\C T_M(w_M))}{(\C T_M(w_M)+1)^\alpha}
					\leq c_0\frac{(\C T_M(w_M)^{1/\sigma}+1)}{(\C T_M(w_M)+1)^\alpha}\leq C.
				$$
		\end{itemize}
		Putting the pieces together, \eqref{eqn:heatEq3} results in
		\begin{align*}
			&\int_\Omega\psi_M(w_M(T))\dx-\int_\Omega \psi_M(w^0)\dx+\widetilde c_1\|\nabla \C T_M(w_M)\|_{L^2(\Omega_T;\R^d)}^2\\
			&\qquad\leq C\left\|\partial_t \chi_M+\rho(\chi_M)\di\bl \partial_t u_M\br
				+\rho'(\chi_M)\di(u_M)\partial_t\chi_M\right\|_{L^1(\Omega_T)}.
		\end{align*}
		The right-hand side estimates as
		\begin{align*}
			\text{r.h.s.}\leq{}&
				C\big(\|\partial_t \chi_M\|_{L^1(\Omega_T)}
				+\|\rho(\chi_M)\|_{L^\infty(\Omega_T)}\|\di\bl \partial_t u_M\br\|_{L^1(\Omega_T)}\\
				&+\|\rho'(\chi_M)\|_{L^\infty(\Omega_T)}\|\di(u_M)\|_{L^2(\Omega_T)}\|\partial_t\chi_M\|_{L^2(\Omega_T)}\big)
		\end{align*}
		and is bounded by the first a~priori estimate.
		
		It remains to show boundedness of $\int_\Omega\psi_M(w_M(T))\dx$.
		Since
		$$
			|\psi_M'(x)|=|(\C T_M(x)+1)^{-\alpha}|\in[0,1],
		$$
		we obtain the growth condition $|\psi_M(x)|\leq |x|$.
		Hence
		$$
			\left|\int_\Omega\psi_M(w_M(T))\dx\right|
				\leq \int_\Omega w_M(T)\dx\quad\text{and}\quad
 			\left|\int_\Omega\psi_M(w^0)\dx\right|
				\leq \int_\Omega w^0\dx.
		$$
		Eventually, we obtain boundedness of $\|\nabla \C T_M(w_M)\|_{L^2(\Omega_T;\R^d)}$.
		The claim follows together with the boundedness of
		$\|\C T_M(w_M)\|_{L^\infty(0,T;L^1(\Omega))}$ by the first a~priori estimate.
		\ep
	\end{proof}\\
%	\begin{lemma}[Third a priori estimate]
%		It holds
%		\begin{align*}
%			&\|u_M\|_{H^1(0,T;H^2(\Omega;\R^d))\cap W^{1,\infty}(0,T;H^1(\Omega;\R^d))\cap H^{2}(0,T;L^2(\Omega;\R^d))}\leq C,
%		\end{align*}
%		for a constant $C>0$ independent of $\tau$.
%	\end{lemma}
	\begin{proof}[Proof of the third a~priori estimate]
		We test \eqref{eqn:momentumEq} with $\zeta=-\di(\DD\e(u_t))$ and adapt a calculation performed in \cite[Sixth a~priori estimate]{RR12}.
		Additionally, we need to estimate
		the following integral term:
		\begin{align*}
		\begin{split}
			&\left|\int_{\Omega_t}\di\bl\rho(\chi_M)\Theta_M(w_M)\mathbf 1\br\cdot\di(\DD\e(\partial_t u_M))\dxs\right|\\
			&\qquad\leq
				\int_{\Omega_t}|(\rho'(\chi_M)\nabla\chi_M\Theta_M(w_M))\cdot\di(\DD\e(\partial_t u_M))|\dxs\\
			&\qquad\quad+\int_{\Omega_t}|\rho(\chi_M)\Theta'(\C T_M(w_M))\nabla\big(\C T_M(w_M)\big)\cdot\di(\DD\e(\partial_t u_M))|\dxs\\
			&\qquad\leq C\|\rho'(\chi_M)\|_{L^\infty(L^\infty)}\|\nabla\chi_M\|_{L^\infty(L^p)}\|\Theta(\C T_M(w_M))\|_{L^2(L^{2p/(p-2)})}
			\|\partial_t u_M\|_{L^2(H^2)}\\
			&\qquad\quad+C\|\rho(\chi_M)\|_{L^\infty(L^\infty)}\|\Theta_M'(w_M)\|_{L^\infty(L^\infty)}\|\nabla\big(\C T_M(w_M)\big)\|_{L^2(L^2)}\|\partial_t u_M\|_{L^2(H^2)}.
		\end{split}
		\end{align*}
		By using the Lipschitz continuity of $\Theta$ (see Assumption (A2))
		and the first as well as the second a~priori estimates,
		it only remains to show boundedness of the term $\|\Theta(\C T_M(w_M))\|_{L^2(L^{2p/(p-2)})}$.
		Indeed, by using the growth assumption in (A2),
		\begin{align}
			\label{eqn:thetaEst}
			\|\Theta(\C T_M(w_M))\|_{L^2(L^{2p/(p-2)})}
				&\leq c_0\bl\|\C T_M(w_M)\|_{L^{2/\sigma}(L^{2p/(\sigma(p-2))})}^{1/\sigma}+1\br.
		\end{align}
		In the case $d=3$, we have $p>3$ and, in particular,
		$2p/(\sigma(p-2))\leq 6$ since $\sigma\geq 3$ by Assumption (A2).
		Consequently, by using the second a~priori estimate, the right-hand side of \eqref{eqn:thetaEst}
		is bounded.
		
		In the cases $d\in\{1,2\}$, boundedness of the right-hand side of \eqref{eqn:thetaEst}
		follows immediately from the second a~priori estimate and $\sigma\geq 3$.
		\ep
	\end{proof}
%	\begin{lemma}[Fourth a priori estimate]
%		It holds
%		\begin{align*}
%			\|\C T_M(w_M)\|_{L^\infty(0,T;L^2(\Omega))\cap L^{2(q+1)}(0,T;L^{6(q+1)}(\Omega))}\leq C.
%		\end{align*}
%	\end{lemma}
	\\
	\begin{proof}[Proof of the fourth a~priori estimate]
		Testing \eqref{eqn:heatEq2} with $\C T_M(w_M)$, integration in time over $[0,t]$ and using the generalized chain-rule yield
		\begin{align*}
			&\int_\Omega\widehat{\C T}_M(w_M(t))\dx-\int_\Omega \widehat{\C T}_M(w^0)\dx
			+\int_{\Omega_t} K(\C T_M(w_M))|\nabla \C T_M(w_M)|^2\dxs\\
			&+\int_{\Omega_t}\bl\partial_t \chi_M+\rho(\chi_M)\di\bl \partial_t u_M\br
				+\rho'(\chi_M)\di(u_M)\partial_t\chi_M\br\Theta_M(w_M)\C T_M(w_M)\dxs=0,
		\end{align*}
		where $\widehat{\C T}_M$ denotes the primitive of $\C T_M$ vanishing at $0$.
		By using Assumption (A3), the estimates (cf. \cite[Remark 2.10]{RR12})
		$$
			c\int_0^t\|\C T_M(w_M)\|_{L^{6(q+1)}(\Omega)}^{2(q+1)}\ds\leq\int_{\Omega_t}(\C T_M(w_M)^{2q}+1)|\nabla \C T_M(w_M)|^2\dxs,
		$$
		and
		$$
			\frac12|\C T_M(w_M)|^2\leq \widehat{\C T}_M(w_M),
		$$
		we obtain by using H\"older's inequality in space and time 
		\begin{align*}
		\begin{split}
			&\int_\Omega\frac12|\C T_M(w_M)|^2\dx-\int_\Omega \widehat{\C T}_M(w^0)\dx
				+\widetilde c\|\C T_M(w_M)\|_{L^{2(q+1)}(0,t;L^{6(q+1)}(\Omega))}^{2(q+1)}\notag\\
			&\qquad
				\leq \|\partial_t\chi_M\|_{L^2(L^2)}\|\Theta_M(w_M)\C T_M(w_M)\|_{L^2(0,t;L^2(\Omega))}\\
			&\qquad\quad+\|\rho(\chi_M)\|_{L^\infty(L^\infty)}\|\di(\partial_t u_M)\|_{L^2(L^2)}\|\Theta_M(w_M)\C T_M(w_M)\|_{L^2(0,t;L^2(\Omega))}\\
			&\qquad\quad+\|\rho'(\chi_M)\|_{L^\infty(L^\infty)}\|\partial_t \chi_M\|_{L^2(L^2)}\|\di(u_M)\|_{L^\infty(L^6)}\|\Theta_M(w_M)\C T_M(w_M)\|_{L^2(0,t;L^3(\Omega))}\\
			&\qquad
				\leq C\|\Theta_M(w_M)\C T_M(w_M)\|_{L^2(0,t;L^3(\Omega))}.
%				\leq\int_0^t\|\partial_t \chi_M+\rho(\chi_M)\di\bl \partial_t u_M\br\|_{L^2(\Omega)}
%				\|\Theta(\C T_M(w_M))\C T_M(w_M)\|_{L^2(\Omega)}\ds\notag\\
%			&\quad\quad+
%				\|\rho'(\chi_M)\|_{L^\infty(\Omega_t)}\|\di(u_M)\|_{L^\infty(0,T;L^6(\Omega))}\int_0^t\|\partial_t\chi_M\|_{L^2(\Omega)}
%				\|\Theta(\C T_M(w_M))\C T_M(w_M)\|_{L^3(\Omega)}\ds.
		\end{split}
		\end{align*}
		Notice the following implications:
		\begin{align*}
			\begin{cases}
				\text{if }0\leq q\leq 1&\text{then }1/\sigma\leq q\text{ (since }2q-1\leq q\text{ and }1/\sigma\leq 2q-1\text{ by (A3))},\\
				\text{if }q>1&\text{then }1/\sigma\leq q\text{ (since }\sigma\geq 3\text{ by (A2)).}
			\end{cases}
		\end{align*}
		Therefore, in both  cases $1/\sigma\leq q$ and we can estimate the right-hand side above as follows
		by using Assumption (A2):
		\begin{align*}
			\|\Theta_M(w_M)\C T_M(w_M)\|_{L^2(0,t;L^3(\Omega))}
			&\leq C\big(\||T_M(w_M)|^{1/\sigma+1}\|_{L^2(0,t;L^3(\Omega))}+1\big)\\
			&\leq C\big(\|\C T_M(w_M)\|_{L^{2(q+1)}(0,t;L^{3(q+1)}(\Omega))}^{q+1}+1\big)\\
			&\leq C\big(\|\C T_M(w_M)\|_{L^{2(q+1)}(0,t;L^{6(q+1)}(\Omega))}^{q+1}+1\big).
%			&\quad
%				\leq \widehat C_\varrho+\varrho\int_0^t\|\C T_M(w_M)\|_{L^{6(q+1)}(\Omega)}^{2(q+1)}\dt.
		\end{align*}
		Thus the r.h.s. can be absorbed by the l.h.s. and we obtain the assertion.\ep
	\end{proof}
%	\begin{lemma}[Fifth a priori estimate]
%		It holds
%		\begin{align*}
%			\|w_M\|_{L^\infty(0,T;L^2(\Omega))\cap L^2(0,T;H^1(\Omega))}\leq C.
%		\end{align*}
%	\end{lemma}
	\\
	\begin{proof}[Proof of the fifth a~priori estimate]
		We test equation \eqref{eqn:heatEq2} with $w_M$, integrate over the time interval $[0,t]$ and obtain
		\begin{align*}
			&\frac 12\int_\Omega|w_M(t)|^2\dx-\frac 12\int_\Omega|w_M(0)|^2\dx
			+\int_{\Omega_t} K_M(w_M)|\nabla w_M|^2\dxt\\
			&\quad+\int_{\Omega_t}\bl\partial_t \chi_M+\rho(\chi_M)\di\bl \partial_t u_M\br
				+\rho'(\chi_M)\di(u_M)\partial_t\chi_M\br\Theta_M(w_M)w_M\dxt=0.
		\end{align*}
		We introduce the sublevel and the strict superlevel set of $w_M(t)$ at height $M$ as
		\begin{subequations}
		\label{eqn:subSupLevelSets}
		\begin{align}
			&l_M^-(t):=\{x\in\Omega\,|\,w_M(x,t)\leq M\},\\
			&l_M^+(t):=\{x\in\Omega\,|\,w_M(x,t)> M\}
		\end{align}
		\end{subequations}
		and receive by utilizing H\"older's inequality as in the fourth a~priori estimate
		\begin{align}
			&\frac 12\int_\Omega|w_M(t)|^2\dx-\frac 12\int_\Omega|w_M(0)|^2\dx
				+c\int_{\Omega_t} |\nabla w_M|^2\dxt
				+c\int_0^t \|w_M\|_{L^{6(q+1)}(l_M^-(s))}^{2(q+1)}\ds\notag\\
			&\quad
				\leq C\|\Theta_M(w_M)w_M\|_{L^2(0,t;L^3(\Omega))}\notag\\
			&\quad
				\leq C\left(\int_0^t\|\Theta_M(w_M)w_M\|_{L^3(l_M^-(s))}^2\ds\right)^{1/2}
				+C\left(\int_0^t\|\Theta_M(w_M)w_M\|_{L^3(l_M^+(s))}^2\ds\right)^{1/2}.
%				\leq C_\varrho\bl\bl\int_0^t\|\Theta(w_M(t))w_M(t)\|_{L^3(l_M^-(t))}^3\dt\br^{1/3}
%				+\bl\int_0^t\|\Theta_M(w_M(t))w_M(t)\|_{L^3(\C O_M)}^3\br^{1/3}\br.
		\label{eqn:ineq1}
		\end{align}
		We treat the last two terms on the right-hand side as follows.
		\begin{itemize}
			\item
				By using the definition $l_M^-(s)$, the growth assumption for $\Theta$ in (A2) and the estimate
				$1/\sigma\leq q$ (see the proof of the fourth a~priori estimate), we obtain
				\begin{align*}
					\int_0^t\|\Theta_M(w_M)w_M\|_{L^3(l_M^-(s))}^2\ds
					&=\int_0^t\|\Theta(w_M)w_M\|_{L^3(l_M^-(s))}^2\ds\\
						&\leq C\bl\int_0^t\|w_M\|_{L^{6(q+1)}(l_M^-(s))}^{2(q+1)}\ds+1\br.
				\end{align*}
%				The r.h.s. can be absorbed by the l.h.s. of \eqref{eqn:ineq1}.
			\item
				H\"older's inequality and the embedding $H^1(\Omega)\hookrightarrow L^6(\Omega)$ yield
				\begin{align*}
					\int_0^t\|\Theta_M(w_M)w_M\|_{L^3(l_M^+(s))}^2\ds
					&\leq \esssup_{t\in(0,T)}\|\Theta_M(w_M(t))\|_{L^6(l_M^+(t))}^2\int_0^t\|w_M\|_{L^6(l_M^+(s))}^2\ds\\
					&\leq \esssup_{t\in(0,T)}\|\Theta_M(w_M(t))\|_{L^6(l_M^+(t))}^2\|w_M\|_{L^2(0,t;H^1(\Omega))}^2.
				\end{align*}
				By the fourth a~priori estimate, we have the boundedness of
				$$
					M^2\esssup_{t\in(0,T)}|l_M^+(t)|=\esssup_{t\in(0,T)}\int_{l_M^+(t)}M^2\dx\,\leq\,\|\C T_M(w_M)\|_{L^\infty(0,T;L^2(\Omega))}^2\dx\leq C,
				$$
				where $|l_M^+(t)|$ denotes the $d$-dimensional Lebesgue measure of $l_M^+(t)$.
				This implies by using the growth condition for $\Theta$ in (A2):
				\begin{align*}
					\esssup_{t\in(0,T)}\|\Theta_M(w_M(t))\|_{L^6(l_M^+(t))}^2&=\Theta(M)^2\esssup_{t\in(0,T)}|l_M^+(t)|^{1/3}\\
						&\leq c_0(M^{2/\sigma}+1)\esssup_{t\in(0,T)}|l_M^+(t)|^{1/3}\\
						&\leq c_0(M^{2/\sigma}+1)\frac{C}{M^{2/3}}.
				\end{align*}
				Since $\sigma\geq 3$, we obtain boundedness of $\esssup_{t\in(0,T)}\|\Theta_M(w_M(t))\|_{L^6(l_M^+(t))}^2$
				and hence
				\begin{align*}
					\int_0^t\|\Theta(w_M)w_M\|_{L^3(l_M^+(s))}^2\ds
					\leq C \|w_M\|_{L^2(0,t;H^1(\Omega))}^2.
				\end{align*}
		\end{itemize}
		Eventually, estimate \eqref{eqn:ineq1} yields to
		\begin{align*}
			&\frac 12\int_\Omega|w_M(t)|^2\dx-\frac 12\int_\Omega|w_M(0)|^2\dx
				+c\|\nabla w_M\|_{L^2(0,t;L^2(\Omega;\R^d))}^2
				+c\int_0^t \|w_M\|_{L^{6(q+1)}(l_M^-(t))}^{2(q+1)}\ds\notag\\
			&\quad
				\leq C\bl\int_0^t\|w_M\|_{L^{6(q+1)}(l_M^-(t))}^{2(q+1)}\ds+1\br^{1/2}
				+C \|w_M\|_{L^2(0,t;H^1(\Omega))}
		\end{align*}
		and thus the claim.\ep
	\end{proof}
%	\begin{lemma}[Sixth a priori estimate]
%		For $r=\frac{2q+2}{2q+1}$ and every $s>3$, there exists a $C>0$ such that
%		\begin{align*}
%			\|w_M\|_{W^{1,r}(0,T;(W_\nu^{2,s}(\Omega))^*)}\leq C
%		\end{align*}
%		for all $M\in\N$.
%	\end{lemma}
	\\
	To tackle the sixth a~priori estimate,
	we will make use of the primitive $\widehat K_M$ of $K_M$ vanishing at $0$
	and use the property
	\begin{align}
	\label{eqn:KMIdentity}
		\widehat K_M(x)=
		\begin{cases}
			\widehat K(x)&\text{if } 0\leq x\leq M,\\
			\widehat K(M)+x-M&\text{if }x>M.
		\end{cases}
	\end{align}
	Note that the identity $\widehat K_M(x)=\widehat K(\C T_M(x))$
	is not fulfilled while $K_M(x)=K(\C T_M(x))$ is true.
	By exploiting growth assumption (A3), we obtain the crucial estimate
	\begin{align}
	\label{eqn:KMestimate}
		\big|\widehat K_M(x)\big|&\leq
		\left.
		\begin{cases}
			C(x^{2q_0+1}+1)&\text{if } 0\leq x\leq M,\\
			C(M^{2q_0+1}+1)+x-M&\text{if }x>M
		\end{cases}
		\right\}
		\leq
		C(\C T_M(x)^{2q_0+1}+1)+x.
	\end{align}
	\begin{proof}[Proof of the sixth a~priori estimate]
		We will use a comparison argument in \eqref{eqn:heatEq2}.
		
		In what follows let $r:=\frac{2q+2}{2q_0+1}$ and $s:=\frac{6q+6}{6q-2q_0+5}$
		as in Definition \ref{def:weakSolution}.
		Applying integration by parts in \eqref{eqn:heatEq2}, we
		receive for all $\zeta\in W_\nu^{2,s}(\Omega)$:
		\begin{align}
			\langle\partial_t w_M,\zeta\rangle
				={}&\int_\Omega\bl \widehat K_M(w_M)\Delta\zeta
				-\bl\Theta_M(w_M)\partial_t \chi_M
				+\rho(\chi_M)\Theta_M(w_M)\di\bl \partial_t u_M\br\br\zeta\br\dx\notag\\
			\label{eqn:wMtDef}
			&-\int_\Omega\rho'(\chi_M)\Theta_M(w_M)\di(u_M)\partial_t\chi_M\zeta\dx.
		\end{align}
		Let
		$s^{**}:=\frac{6q+6}{2q-2q_0+1}>0$ denote the constant
		resulting from the continuous embedding
		$W^{2,s}(\Omega)\hookrightarrow L^{s^{**}}(\Omega)$.
		Due to the crucial identities
		\begin{align}
		\label{eqn:HoelderExp}
			\frac{1}{\frac{6q+6}{2q_0+1}}+\frac{1}{\frac{6q+6}{6q-2q_0+5}}=1,\;
			\frac{1}{\frac{6q+6}{2q_0+1}}+\frac{1}{6}+\frac{1}{2}+\frac{1}{\frac{6q+6}{2q-2q_0+1}}=1\text{ and }
			\frac{1}{\frac{6q+6}{q+2q_0+2}}+\frac{1}{2}+\frac{1}{\frac{6q+6}{2q-2q_0+1}}=1,
		\end{align}
		H\"older's inequality reveals
		\begin{align*}
		\begin{split}
			\langle\partial_t w_M,\zeta\rangle
				\leq{}&\|\widehat K_M(w_M)\|_{L^{\frac{6q+6}{2q_0+1}}}\|\Delta\zeta\|_{L^s}
				 +\|\Theta_M(w_M)\|_{L^\frac{6q+6}{q+2q_0+2}}\|\partial_t \chi_M\|_{L^2}\|\zeta\|_{L^{s^{**}}}\\
				&+\|\rho(\chi_M)\|_{L^\infty}\|\Theta_M(w_M)\|_{L^\frac{6q+6}{q+2q_0+2}}\|\di\bl \partial_t u_M\br\|_{L^2}\|\zeta\|_{L^{s^{**}}}\\
				&+\|\rho'(\chi_M)\|_{L^\infty}\|\Theta_M(w_M)\|_{L^\frac{6q+6}{2q_0+1}}\|\di(u_M)\|_{L^6}\|\partial_t\chi_M\|_{L^2}\|\zeta\|_{L^{s^{**}}}.
		\end{split}
		\end{align*}
		
		By using the boundedness of $\chi_M$ in $L^\infty(0,T;L^\infty(\Omega))$
		and $\frac{6q+6}{q+2q_0+2}\leq \frac{6q+6}{2q_0+1}$, we obtain
		\begin{align*}
		\begin{split}
			\|\partial_t w_M\|_{(W_\nu^{2,s})^*}
			\leq{}&C\|\Theta_M(w_M)\|_{L^\frac{6q+6}{2q_0+1}}\Big(
					\|\partial_t \chi_M\|_{L^2}
					+\|\di\bl \partial_t u_M\br\|_{L^2}\\
					&+\|\di(u_M)\|_{L^6}\|\partial_t\chi_M\|_{L^2}\Big)
				+C\|\widehat K_M(w_M)\|_{L^{\frac{6q+6}{2q_0+1}}}.
		\end{split}
		\end{align*}
%		Here, we have used the calculations (using (A2) and (A3))
%		$$
%			s^{**}\geq\frac{6\sigma}{2\sigma-1}\text{ and }\frac12+\frac16+\frac{1}{6\sigma}+\frac{1}{\frac{6\sigma}{2\sigma-1}}=1.
%		$$

%		Note that, for $s>3$, we obtain the continuous embedding $L^1(\Omega)\hookrightarrow W^{2,s}(\Omega)^*$.
		Calculating the $L^r$-norm in time and using H\"older's inequality show
		\begin{align*}
		\begin{split}
			&\|\partial_t w_M\|_{L^r\big((W_\nu^{2,s})^*\big)}\\
			&\qquad\leq
				C\|\Theta_M(w_M)\|_{L^\frac{2r}{2-r}\big(L^\frac{6q+6}{2q_0+1}\big)}\Big(\|\partial_t \chi_M\|_{L^2(L^2)}
				+\|\di\bl \partial_t u_M\br\|_{L^2(L^2)}\\
			&\qquad\quad+\|\di(u_M)\|_{L^\infty(L^6)}\|\partial_t\chi_M\|_{L^2(L^2)}\Big)
			+C\|\widehat K_M(w_M)\|_{L^r\big(L^{\frac{6q+6}{2q_0+1}}\big)}.
		\end{split}
		\end{align*}
		Keeping the first and the third a~priori estimates in mind,
		it still remains to show
		\begin{subequations}
		\begin{align}
		\label{eqn:KMbound}
			&\{\widehat K_M(w_M)\}\text{ bounded in }L^r\big(0,T;L^{\frac{6q+6}{2q_0+1}}(\Omega)\big),\\
		\label{eqn:ThetaMbound}
			&\{\Theta_M(w_M)\}\text{ bounded in }L^\frac{2r}{2-r}\big(0,T;L^\frac{6q+6}{2q_0+1}(\Omega)\big).
		\end{align}
		\end{subequations}
		\begin{itemize}
			\item
				Estimate \eqref{eqn:KMestimate} leads to
				\begin{align}
				\label{eqn:KMTerm}
					\|\widehat K_M(w_M)\|_{L^r\big(L^{\frac{6q+6}{2q_0+1}}\big)}
					\leq C(\|\C T_M(w_M)\|_{L^{r(2q_0+1)}(L^{6q+6})}^{2q_0+1}+1)+\|w_M\|_{L^r\big(L^{\frac{6q+6}{2q_0+1}}\big)}.
				\end{align}
%		We adopt the notation
%		in \eqref{eqn:subSupLevelSets} and use the $(2q+1)$-growth of $\widehat K_M$ on $l_M^-(t)$
%		and the linear growth of $\widehat K_M$ on $l_M^+(t)$ (see above). More precisely,
%		\begin{align*}
%			\|\widehat K_M(w_M)\|_{L^r(L^{s'})}^r 
%				={}&\int_0^T\|C(w_M(t)^{2q+1}+1)\|_{L^{s'}(l_M^-(t))}^r\dt\\
%			&+\int_0^T\|C(M^{2q+1}+1)+w_M(t))\|_{L^{s'}(l_M^+(t))}^r\dt\\
%			\leq{}&\widetilde C\bl\int_0^T\|\C T_M( w_M(t))\|_{L^{s'(2q+1)}(l_M^-(t))}^{r(2q+1)}\dt+1\br\\
%			&+\widetilde C\bl\int_0^T\bl\|\C T_M( w_M(t))\|_{L^{s'(2q+1)}(l_M^+(t))}^{r(2q+1)}+\|w_M(t)\|_{L^{s'}(l_M^+(t))}^r\br\dt+1\br\\
%			\leq{}&\widetilde C\bl\|\C T_M(w_M)\|_{L^{r(2q+1)}(L^{s'(2q+1)})}^{2q+1}+\|w_M\|_{L^r(L^{s'})}^r+2\br.
%		\end{align*}
				Since, by definition, $r(2q_0+1)=2(q+1)$, we infer boundedness of
				$$
					\{\C T_M(w_M)\}\text{ in }L^{r(2q_0+1)}(0,T;L^{6q+6}(\Omega))
				$$
				by the fourth a~priori estimate
				and boundedness of
				$$
					\{w_M\}\text{ in }L^r\big(0,T;L^{\frac{6q+6}{2q_0+1}}(\Omega)\big)
				$$
				by the fifth a~priori estimate and by $r\in(1,2)$ and $\frac{6q+6}{2q_0+1}\leq 6$ using (A3).
				Finally, we obtain \eqref{eqn:KMbound}.
			\item
				By Assumption (A2), we obtain
				$$
					\|\Theta_M(w_M)\|_{L^{\frac{2r}{2-r}}\big(0,T;L^\frac{6q+6}{2q_0+1}(\Omega)\big)}
					\leq C(\|w_M\|_{L^{\frac{2r}{(2-r)\sigma}}\big(0,T;L^{\frac{6q+6}{(2q_0+1)\sigma}}(\Omega)\big)}^{1/\sigma}+1).
				$$
				Because of $\frac{6q+6}{(2q_0+1)\sigma}\leq 2$ (since $\sigma\geq 3$ by (A2) and $q\leq q_0$ by (A3)),
				we obtain \eqref{eqn:ThetaMbound} by the fifth a~priori estimate.
		\end{itemize}
		\ep
	\end{proof}
	
	\subsection{The passage $M\uparrow\infty$}
	The a~priori estimates from Lemma \ref{lemma:aPriori} give rise to the subsequent convergence properties
	for $\{u_M\}$, $\{w_M\}$ and $\{\chi_M\}$ along subsequences by Aubin-Lions type compactness results
	(cf. \cite{Simon})
	and by adapting Lemma \ref{lemma:strongZConv} to this case.
	\begin{corollary}
	\label{cor:convergenceLimit}
		There exist limit functions $(u,w,\chi)$ defined in spaces given in Definition \ref{def:weakSolution}
		such that the following convergence properties are satisfied for all $\mu\geq 1$, $s>3$
		and all $\varepsilon\in(0,1]$ (as $M\uparrow\infty$ for a subsequence):
		\begin{align*}
			\textit{(i) }&u_M\to u&&\textit{ weakly-star in }H^1(0,T;H^2(\Omega;\R^d))\cap W^{1,\infty}(0,T;H^1(\Omega;\R^d))\hspace*{4.3em}\\
				&&&\hspace*{6.6em}\cap H^{2}(0,T;L^2(\Omega;\R^d)),\\
				&u_M\to u&&\textit{ strongly in }H^1(0,T;H^{2-\varepsilon}(\Omega;\R^d)),\\
				&u_M\to u&&\textit{ a.e. in }\Omega_T,\\
			\textit{(ii) }&w_M\to w&&\textit{ weakly-star in }L^2(0,T;H^1(\Omega))\cap L^\infty(0,T;L^2(\Omega))\\
				&&&\hspace*{6.6em}\cap W^{1,r}(0,T;W_\nu^{2,s}(\Omega)^*),\\
				&w_M\to w&&\textit{ strongly in }L^2(0,T;H^{1-\varepsilon}(\Omega))\cap L^\mu(0,T;L^2(\Omega)),\\
				&w_M\to w&&\textit{ a.e. in }\Omega_T,\\
%				$\C T_M(w_M)\to w$ weakly in $L^{2(q+1)}(0,T;L^{6(q+1)}(\Omega))$,
			\textit{(iii) }&\chi_M\to\chi&&\textit{ weakly-star in }L^\infty(0,T;W^{1,p}(\Omega))\cap H^1(0,T;L^2(\Omega)),\\
				&\chi_M\to\chi&&\textit{ strongly in }L^\mu(0,T;W^{1,p}(\Omega)),\\
				&\chi_M\to\chi&&\textit{ uniformly on }\OL{\Omega_T}.
		\end{align*}
	\end{corollary}
	We refer to the comments after Corollary \ref{cor:weakConvDiscr} to indicate how the strong convergence properties can be achieved.
	
	Corollary \ref{cor:convergenceLimit} can be used to prove convergence of $\widehat K_M(w_M)$,
	$\Theta_M(w_M)$ and $\xi_M$ as $M\uparrow\infty$ in suitable spaces.
	More precisely, we obtain the following result.
	\begin{corollary}
	\label{cor:convergenceLimit2}
		There exists an element $\xi\in L^2(0,T;L^{2}(\Omega))$ such that for all $1\leq \lambda<6$
		(as $M\uparrow\infty$ for a subsequence):
		\begin{align*}
			\textit{(i) }&
				\widehat K_M(w_M)\to \widehat K(w)&&\textit{ weakly in }L^{\frac{2q+2}{2q_0+1}}\big(0,T;L^{\frac{6q+6}{2q_0+1}}\big(\Omega)),\hspace*{11.3em}\\
			\textit{(ii) }&
				\Theta_M(w_M)\to \Theta(w)&&\textit{ strongly in }L^{2\sigma}(0,T;L^{\lambda\sigma}(\Omega)),\\
			\textit{(iii) }&
				\xi_M\to\xi&&\textit{ weakly in }L^2(0,T;L^{2}(\Omega)).
		\end{align*}
	\end{corollary}
	\begin{proof}
%		For the remaining items, we can argue as follows.
		\begin{itemize}
		 \item[(i)]
				We obtain the estimate
				\begin{align*}
					&\|\widehat K_M(w_M)\|_{L^{\frac{2q+2}{2q_0+1}}\big(L^{\frac{6q+6}{2q_0+1}}\big)}\\
					&\qquad\leq C(\|\C T_M(w_M)\|_{L^{2q+2}(L^{6q+6})}^{2q_0+1}+1)+\|w_M\|_{L^{\frac{2q+2}{2q_0+1}}\big(L^{\frac{6q+6}{2q_0+1}}\big)}.
				\end{align*}
				due to \eqref{eqn:KMestimate}.
				The first summand on the right-hand side is bounded by the fourth a~priori estimate while
				the second one is bounded by the fifth a~priori estimate.
				
				This enables us to choose a subsequence (we omit the subindex) such that
				\begin{align}
				\label{eqn:KMWeakConv}
					\widehat K_M(w_M)\to\eta\text{ weakly in }L^{\frac{2q+2}{2q_0+1}}\big(0,T;L^{\frac{6q+6}{2q_0+1}}\big(\Omega))
				\end{align}
				for an element $\eta\in L^{\frac{2q+2}{2q_0+1}}\big(0,T;L^{\frac{6q+6}{2q_0+1}}\big(\Omega))$.
				
				Furthermore, noticing that $w_M\to w$ a.e. in $\Omega_T$ as $M\uparrow\infty$, we conclude
				\begin{align}
				\label{eqn:KMPointwiseConv}
					\widehat K_M(w_M)\to \widehat K(w)\text{ a.e. in }\Omega_T.
				\end{align}
				From \eqref{eqn:KMWeakConv} and \eqref{eqn:KMPointwiseConv} we conclude (i).
			\item[(ii)]
				This item follows from the fact that $w_M\to w$ converges strongly in $L^2(0,T;L^{\lambda}(\Omega))$ for all $1\leq \lambda<6$
				and from the growth condition for $\Theta$ in (A2).
			\item[(iii)]
				By referring to the construction of $\xi_M$ in \eqref{eqn:defXiM}, we choose a weakly-star cluster point for the sequence $\{\mathbf 1_{\{\chi_M=0\}}\}_{M\in\N}$
				(here $\mathbf 1_{\{\chi_M=0\}}:\Omega_T\to\{0,1\}$ denotes the characteristic function on the level set $\{\chi_M=0\}$),
				i.e.
				$$
					\pi_M:=\mathbf 1_{\{\chi_M=0\}}\to \pi\text{ weakly-star in }L^\infty(0,T;L^\infty(\Omega))
				$$
				as $M\uparrow\infty$ for a subsequence.
				By the already known convergence properties, we deduce that
				the sequence of functions
				$$
					\eta_M:=\Big(\gamma(\chi_M)+\frac{b'(\chi_M)}{2}|\e(u_M)|^2-\Theta_M(w_M)-\rho'(\chi_M)\Theta_M(w_M)\di(u_M)\Big)^+
				$$
				converges strongly to the corresponding limit function $\eta$ in $L^2(0,T;L^2(\Omega))$.
				This proves
				$$
					\xi_M=-\pi_M\eta_M\to -\pi\eta=:\xi\text{ weakly in }L^2(0,T;L^2(\Omega))
				$$
				as desired.
		\end{itemize}
		\ep
	\end{proof}
	\begin{remark}
			We remark that the weakly-star limit function $\pi$ is not necessarily of the form $\mathbf 1_{\{\chi=0\}}$
			and, therefore, $\xi$ need not to be of the form $-\eta\mathbf 1_{\{\chi=0\}}$
			(in contrast to the truncated system; see \eqref{eqn:defXiM}).
			What only matters is that the weak convergence properties of $\xi_M$ together with the properties for $(u_M,w_M,\chi_M)$ suffices to pass to the limit 
			in the variational inequalities \eqref{eqn:damageEq} and \eqref{eqn:damageEq2} in order to obtain \eqref{eqn:weak3} and \eqref{eqn:weak4} as indicated below.
	\end{remark}
	\begin{proof}[Proof of Theorem \ref{theorem:existence}]
	The limit passage of the truncated system given by \eqref{eqn:heatEq2}, \eqref{eqn:momentumEq},
	\eqref{eqn:damageEq}, \eqref{eqn:damageEq2} and \eqref{eqn:EI} as $M\uparrow\infty$
	can now be performed with Corollary \ref{cor:convergenceLimit} and Corollary \ref{cor:convergenceLimit2}.
	\begin{itemize}
		\item
			\textbf{Heat equation.}
			Integrating \eqref{eqn:heatEq2} in time and applying integration by parts show
			\begin{align*}
				&\int_0^T\langle\partial_t w_M,\Psi\rangle\dt\\
				&\qquad-\int_{\Omega_T}\bl \widehat K_M(w_M)\Delta\Psi
					-\bl\Theta_M(w_M)\partial_t \chi_M
					+\rho(\chi_M)\Theta_M(w_M)\di\bl \partial_t u_M\br\br\Psi\br\dxt\\
				&\qquad+\int_{\Omega_T}\rho'(\chi_M)\Theta_M(w_M)\di(u_M)\partial_t\chi_M\Psi\dxt=0,
			\end{align*}
			for all test-functions $\Psi\in C([0,T];W_\nu^{2,s}(\Omega))$.
			Taking \eqref{eqn:HoelderExp} into account, passing $M\uparrow\infty$ by
			employing the convergence results in Corollary \ref{cor:convergenceLimit} and Corollary \ref{cor:convergenceLimit2}
			and switching back to an a.e. in time formulation,
			we end up with \eqref{eqn:weak1}.
		\item
			\textbf{Balance of momentum equation} and \textbf{one-sided variational inequality.}
			Translating \eqref{eqn:momentumEq}, \eqref{eqn:damageEq} and \eqref{eqn:damageEq2} to a weak formulation involving
			test-functions in time and space, we can pass $M\uparrow\infty$.
			Translating the results back to an a.e. in time formulation, we obtain \eqref{eqn:weak2}, \eqref{eqn:weak3} and \eqref{eqn:weak4}.
		\item
			\textbf{Partial energy inequality.}
			The inequality \eqref{eqn:weak5} is gained from \eqref{eqn:EI} by using lower semi-continuity arguments
			in the transition $M\uparrow\infty$.\ep
	\end{itemize}
	\end{proof}
	
{\em Acknowledgement.} \,\,We thank Riccarda Rossi from the University of Brescia and the anonymous referees for their helpful
comments on the first version of this paper.

	\addcontentsline{toc}{chapter}{Bibliography}{\footnotesize{\setlength{\baselineskip}{0.2 \baselineskip}
	\bibliography{references}
	}
	\bibliographystyle{plain}}
\end{document}